\documentclass[12pt]{amsart}
\usepackage{amsmath, amssymb, amsthm}
\usepackage{comment, empheq}
\usepackage{mathdots, breqn}
\usepackage{graphicx}
\usepackage{hyperref}
\usepackage{mathtools}
\numberwithin{equation}{section}
\newtheoremstyle{myremark}{10pt}{10pt}{}{}{\scshape}{.}{.5em}{}
\newtheorem{theorem}{Theorem}

\newtheorem{lemma}{Lemma}[section]
\newtheorem{cor}{Corollary}
\newtheorem{prop}{Proposition}[section]

\theoremstyle{remark}
\theoremstyle{myremark}
\newtheorem{remark}{Remark}
\newtheorem*{ackno}{Acknowledgements}

\newcommand{\R}{\mathbb{R}}
\newcommand{\X}{\mathbb{X}}
\newcommand{\C}{\mathbb{C}}
\newcommand{\Q}{\mathbb{Q}}
\newcommand{\Z}{\mathbb{Z}}

\newcommand{\A}{\mathbb{A}}
\newcommand{\B}{\mathcal{B}}
\newcommand{\D}{\mathcal{D}}
\newcommand{\W}{\mathcal{W}}
\newcommand{\GL}{\mathrm{GL}}

\newcommand{\I}{\mathcal{I}}
\newcommand{\F}{\mathcal{F}}
\newcommand{\PGL}{\mathrm{PGL}}
\newcommand{\SL}{\mathrm{SL}}

\newcommand{\loc}{\mathrm{loc}}
 
\newcommand{\Ad}{\mathrm{Ad}}
\newcommand{\gen}{\mathrm{gen}}
\newcommand{\diag}{\mathrm{diag}}
\newcommand{\aut}{\mathrm{aut}}
\newcommand{\Eis}{\mathrm{Eis}}
\usepackage[usenames,dvipsnames]{color}

\title{The second moment of $\mathrm{GL}(n)\times\mathrm{GL}(n)$ Rankin--Selberg $L$-functions}
\author{Subhajit Jana}
\address{Max Planck Institute for Mathematics, Vivatsgasse 7, 53111 Bonn, Germany.}
\email{subhajit@mpim-bonn.mpg.de}

\begin{document}

\begin{abstract}
We prove an asymptotic expansion of the second moment of the central values of the $\GL(n)\times\GL(n)$ Rankin--Selberg $L$-functions $L(1/2,\pi\otimes\pi_0)$, for a fixed cuspidal automorphic representation $\pi_0$, over the family of $\pi$ with analytic conductors bounded by a quantity that is tending to infinity. Our proof uses the integral representations of the $L$-functions, period with regularized Eisenstein series, and the invariance properties of the analytic newvectors.
\end{abstract}


\maketitle

\section{Introduction}

The asymptotic evaluation of higher moments of the central $L$-values carries important arithmetic information, e.g.\ subconvex bounds, or non-vanishing result for the central $L$-values. This evaluation becomes more and more difficult as the moment, the degrees of the $L$-functions, or the rank of the underlying group increase. 

Obtaining subconvex bounds i.e.\ proving bounds of the form
$$L(1/2,\pi)\ll C(\pi)^{1/4-\delta},\quad \delta>0\footnote{The generalized Lindel\"of hypothesis predicts that any $\delta<1/4$ is achievable.},$$
where $C(\pi)$ is the \emph{analytic conductor} of an automorphic representation $\pi$, is an extremely difficult problem with respect to the current technology. A narrow, but important for applications, class of automorphic representations suffers from yet another major technical difficulty, named \emph{conductor-drop}. These representations are usually functorially lifted from smaller groups and have unusually small analytic conductors. 

For example, if $\pi$ varies over automorphic representations for $\PGL_2(\Q)$ with $C(\pi)$ being of size $T$ then the size of the analytic conductor of the Rankin--Selberg convolution $\pi\otimes\tilde{\pi}$ is roughly $T^2$ where $\tilde{\pi}$ is the contragredient of $\pi$, whereas $C(\pi\otimes \pi')$ has size $T^4$ if $C(\pi')$ is of size $T^2$ but $\pi'$ is \emph{away} from $\tilde{\pi}$. That is, the $\PGL(4)$-subfamily of $\pi\otimes\tilde{\pi}$ shows the conductor-drop phenomena. Another example of a family that sees conductor-dropping is the $\PGL(3)$-family of $\mathrm{Sym}^2\pi$ where $\pi$ varies over a $\PGL(2)$ family (the subconvexity problem for this family is directly related to the arithmetic quantum unique ergodicity problem for $\SL_2(\R)$). This happens due to one of the Langlands parameter of $\mathrm{Sym}^2\pi$ being extremely small compared to the others. The families defined by the Plancherel balls with large radius (e.g.\ dilated) or high center often exclude these narrow classes. Thus moment estimates over these families does not usually become fruitful to yield a subconvex bound of an $L$-function that has conductor drop; see e.g.\ \cite{BB, NV, N2}.

\vspace{0.5cm}

One naturally interesting and important family of automorphic representations can be given by representations with growing conductors, e.g.\
$$\F_X:=\{\pi \text{ automorphic representation for }\PGL_n(\Z)\mid C(\pi)<X\},$$
with $X\to\infty$.
The family $\F_X$, unlike the families defined by the Plancherel balls, is indifferent towards the conductor-drop issue. So a Lindel\"of-consistent estimate for a high enough moment over the family $\mathcal{F}_X$ will likely produce a subconvex estimate even for the $L$-functions suffering from conductor-drop. Here, by \emph{Lindel\"of-consistent} (also called \emph{Lindel\"of on average}) estimate for the $2k$-th moment we mean the estimate
$$\mathbb{E}_{\mathcal{F}_X}|L(1/2,\pi)|^{2k}\ll_\epsilon X^{\epsilon},$$
where $\mathbb{E}$ denotes the average. On the other hand, a more interesting and difficult question would be to find an asymptotic formula of (a suitably weighted and smoothened version of) the above average whose leading term is believed to be a polynomial in $\log X$.

However, the family $\F_X$ becomes quite large as $X$ tends to infinity. One informally has $|\F_X|\asymp X^{n-1}$; see \cite{BrMi} for the corresponding non-archimedean analogue. This is why to obtain a subconvex bound of an $L$-function attached to an element in $\F_X$ one needs to evaluate quite a high moment asymptotically, or at least, estimate in the Lindel\"of-consistent manner. For example, we need to estimate an amplified $4(n-1)$-th moment over $\F_X$ even to break the convexity barrier. Unfortunately, the current technology is not advanced enough to tackle such a high moment of these $L$-functions due to the large size of the conductors. Hence a natural informal question arises regarding the race between the sizes of the conductors of the $L$-functions and the families: \emph{as a function of $n$ how high of a moment can be asymptotically evaluated (or estimated in a Lindel\"of-consistent manner) over the family $\F_X$}?

Such a question has been addressed in the literature for low rank groups.
We may try to guess an answer to our proposed informal question by looking at the small number of examples in low ranks. For $n=2$ in \cite{KMV1} the authors obtained an asymptotic formula of the $4$-th moment over a family in the non-archimedean conductor aspect and restricted only to the holomorphic forms. In \cite{BBM} the authors proved a Lindel\"of-consistent upper bound of the $6$-th moment in the non-archimedean conductor aspect for $n=3$. These are the best possible estimates so far for small $n$, which allows us wonder whether the $2n$-th moment can be asymptotically evaluated over the family $\F_X$. However, if we work on $\GL(n)$ rather than on $\PGL(n)$ family, that is, if we do an extra central average we expect that an asymptotic formula of the $2n+2$-th moment is achievable.

\vspace{0.5cm}

Our primary motivation is to prove an asymptotic formula for the $2n$-th moment of the central $L$-values for $\PGL(n)$ with $n\ge 3$, over the family $\F_X$, using the integral representations of the $L$-functions and spectral theory.\ If $\pi$ is an automorphic representation for $\PGL(n)$ then 
$$L(1/2,\pi)^n= L(1/2,\pi \otimes E_0),$$
where $E_0$ is the minimal Eisenstein series for $\PGL(n)$ with trivial Langlands parameters and $\otimes$ denotes the Rankin--Selberg convolution. Thus to evaluate the $2n$-th moment of $L(1/2,\pi)$ it is same to evaluate the second moment of $L(1/2,\pi\otimes E_0)$. However, the approach of the integral representations and the spectral decomposition encounters severe analytic difficulties due to the growth of $E_0$ near the cusp, e.g.\ $E_0$ fails to be square integrable in the fundamental domain. To avoid this particular technical difficulty we may replace $E_0$ by a fixed cusp form, and try to evaluate their second moment asymptotically. 

Let $n\ge 3$. In this article we evaluate the second moment of the central Rankin--Selberg $L$-values $L(1/2,\pi\otimes\pi_0)$ where $\pi$ varies over a family of automorphic representations for $\PGL_n(\Q)$ that are unramified at all the finite places and the archimedean conductors are growing to infinity. Here $\pi_0$ is a fixed cuspidal representation for $\PGL_n(\Q)$, which is again unramified at all the finite places.
Below we informally describe our main theorem.

\begin{theorem}[Informal version]\label{informal-theorem}
Let $n\ge 3$ and $\pi_0$ be a cuspidal automorphic representation for $\PGL_n(\Z)$ (i.e.\ unramified at the finite places), which is tempered at $\infty$. Let $\pi$ vary over the generic automorphic representations in $\F_X$. Then we have an asymptotic formula of the following (weighted) average
$$\mathbb{E}_{\substack{{\pi\in\F_X}\\{\mathrm{generic}}}}\left[\frac{|L(1/2,\pi\otimes\pi_0)|^2}{L(1,\pi,\Ad)}+\mathrm{continuous}\right]=  n\frac{\zeta(n/2)^2}{\zeta(n)}L(1,\pi_0,\Ad)\log X + O_{\pi_0}(1),$$
as $X$ tends to infinity.
\end{theorem}

For the actual formal statement we refer to Theorem \ref{main-theorem}. 

\begin{remark}
In Theorem \ref{informal-theorem} by ``continuous'' we mean the corresponding terms from the generic non-cuspidal spectrum. In the actual statement i.e.\ Theorem \ref{main-theorem}, we do a specific weighted average over the full generic automorphic spectrum such that the weights are uniformly bounded away from zero on the cuspidal spectrum with analytic conductors bounded by $X$. Consequently, we also need to change the harmonic weight $L(1,\pi,\Ad)$ by an equivalent arithmetic factor for the non-cuspidal spectrum.
\end{remark}

This is the first instance of an asymptotic evaluation of the second moment of a family of $L$-functions with arbitrary high degree. 
In general, for a pair of groups $H\le G$ and their representations $\pi$ and $\Pi$ respectively, it is an interesting question to asymptotically evaluate moments of the central $L$-values of the Rankin--Selberg product $\Pi\otimes\pi$ (if defined).
Previously, in \cite{NV} Nelson--Venkatesh asymptotically evaluated the first moment keeping $\Pi$ fixed and letting $\pi$ vary over a dilated Plancherel ball when $(G,H)$ are  Gan--Gross--Prasad pairs, more interestingly, allowing arbitrary weights in the spectral side. More recently in \cite{N2} Nelson proved a Lindel\"of-consistent upper bound of the first moment for the groups $(G,H)=(\mathrm{U}(n+1),\mathrm{U}(n))$ in the non-split case keeping $\pi$ fixed and letting $\Pi$ vary over a Plancherel ball with high center. Both \cite{NV,N2} assume that the family of $L(s,\Pi\otimes\pi)$ does not show any conductor-dropping. The method in \cite{N2} also yields an asymptotic formula with power savings of a specific weighted first moment over this family.
Blomer in \cite{B3} obtained a Lindel\"of-consistent upper bound of the second moment for $G=H=\GL(n)$ keeping $\Pi$ a fixed cuspidal representation and letting $\pi$ vary in a Plancherel ball. On the contrary to \cite{NV,N2}, he proves a Lindel\"of-consistent upper bound when the family of $L(s,\Pi\otimes\pi)$ \emph{shows} conductor-dropping. However, his method does not yield an asymptotic formula.

There have been quite a few results for asymptotic formula and upper bounds on rank $\le 2$ and degree $\le 4$. In particular, we refer to \cite{BH} where the authors prove an asymptotic formula for $\GL(2)\times\GL(2)$ Rankin--Selberg $L$-functions fixing one of the representations but with an extra average over the center of $\GL(2)$. In \cite{CL} an asymptotic formula for the sixth moment of the $L$-values attached to holomorphic cusp forms for $\GL(2)$ is achieved but again with an extra averages over the center of $\GL(2)$.

\subsection{Sketch for the proof}

Our point of departure is similar to \cite{B3} and \cite{V}. We use the spectral decomposition of $\PGL(n)$ and integral representations of the $L$-functions. We start by choosing $\phi_0\in\pi_0$ such that the Whittaker function $W_0$ of $\phi_0$ is an \emph{analytic newvector}; see \S\ref{analytic-newvector} for a brief description of the analytic newvectors. Such $W_0$ in the Kirillov model of $\pi_0$ can be described by a fixed bump function. Let $\Eis(f_s)$ be the maximal Eisenstein series $\PGL_n(\Z)$ attached to a generalized principal series vector $f_s$. Also let $X$ be a large real number and $x$ be the diagonal element in $\PGL_n(\R)$ given by $\diag(X,\dots,X,1)$.
We translate the Eisenstein series by $x$ to obtain $\Eis(f_s)(.x)$.

For this subsection let $\X:=\PGL_n(\Z)\backslash\PGL_n(\R)$ and $N$ be the maximal unipotent of the upper triangular matrices in $\PGL_n(\R)$. We start by writing the inner product
\begin{equation}\label{sketch-main-period}
    \langle \phi_0\Eis(f_{1/2})(.x),\phi_0\Eis(f_{1/2})(.x)\rangle=\langle |\phi_0|^2, |\Eis(f_{1/2})|^2(.x)\rangle.
\end{equation}
where all the inner products above are the usual $L^2$-inner product on the fundamental domain $\X$. Note that both of the sides of \eqref{sketch-main-period} are absolutely convergent as $\phi_0$ decays rapidly at the cusps.

We use Parseval's identity on the left hand side over $\PGL(n)$. A typical term corresponding to an automorphic representation $\pi$ in the spectral sum would look like
$$\left|\int_\X \phi_0(g)\overline{\phi(g)}\Eis(f_{1/2})(gx)dg\right|^2 = \frac{|L(1/2,\tilde{\pi}\otimes\pi_0|^2}{L(1,\pi,\Ad)}|Z_x(f_{1/2},W,W_0)|^2,$$
and
$$Z_x(f_{s},W,W_0)=\int_{N\backslash\PGL_r(\R)} W_0(g)\overline{W(g)}f_{s}(gx)dg$$
is the local zeta integral.

We choose $f_s$ such that $f_{1/2}\left[\begin{pmatrix}I&\\c&1\end{pmatrix}\right]$ is supported on $|c|<\tau$ for some $\tau>0$ sufficiently small, so that $W_0f_{1/2}(.x)$ would mimic a smoothened characteristic function of the archimedean congruence subgroup $K_0(X,\tau)$ (see \eqref{defn-congruence-subgroup}). If $W$ is an \emph{analytic newvector} (see \S\ref{analytic-newvector}) then the invariance property of $W$ would yield that $Z_x(f_{1/2},W,W_0)\gg 1$ if $C(\pi)<X$. We use 
$$\sum_{W}|Z_x(f_{1/2},W,W_0)|^2$$
as the spectral weights where in the above sum $W$ traverses some orthonormal basis of $\pi$. We point out on the naive similarities between the spectral weight here and the one that is used in e.g.\ \cite[Theorem 1]{J20}. However, the invariance property that is needed here is a bit stronger than the invariance used in \cite[Theorem 1]{J20}: we only needed invariance at points near the identity in $\GL_n(\R)$ in \cite{J20}, whereas here we have to gain an invariance that is uniform for all elements in $\GL_{n-1}(\R)$. The method of using the \emph{approximate} invariance of the newvectors is similar as in \cite{V} for $\GL(2)$ where in the non-archimedean aspect the exact invariance is used. This analysis is done in \S\ref{analysis-spectral-side}.

We now explain how we proceed to give an asymptotic expansion of the right hand side of \eqref{sketch-main-period}. The heuristic idea, at least to obtain an upper bound, is to make the change of variables in the period of the right hand side of \eqref{sketch-main-period} to write it as 
$$\int_\X|\phi_0(gx^{-1})|^2|\Eis(f_{1/2})(g)|^2dg,$$
and then bound this period by 
$$\le \|\phi_0\|^2_{L^\infty(\X)}\int_{\X}|\Eis(f_{1/2})(g)|^2dg.$$
But unfortunately, $\Eis(f_{1/2})$ (barely!) fails to be square integrable on $\X$. That is why we have to regularize the period. We adopt regularizing techniques of Zagier \cite{Z}, see also \cite{MV,N1}. First we deform $|\Eis(f_{1/2})|^2$ as $\Eis(f_{1/2+s})\overline{\Eis(f_{1/2})}$ for $s$ lying in some generic position with very small $\Re(s)$. From the Fourier expansions of the Eisenstein series we can pick off the non-integrable terms in the product $\Eis(f_{1/2+s})\overline{\Eis(f_{1/2})}$ and call their sum to be $F_s$. Then we construct a regularized Eisenstein series by
\begin{equation*}
    \tilde{E}(s,.):=\Eis(f_{1/2+s})\overline{\Eis(f_{1/2})}-\Eis(F_s).
\end{equation*}
We will check that $\tilde{E}(s,.)$ lies in $L^2(\X)$.
Consequently, we regularize the period as 
\begin{align*}
    \langle |\phi_0(.x^{-1})|^2, |\Eis(f_{1/2})|^2\rangle
    &=\lim_{s\to 0}\langle |\phi_0(.x^{-1})|^2, \overline{\Eis(f_{1/2})}\Eis(f_{1/2+s})\rangle\\
    &=\lim_{s\to 0}\langle |\phi_0(.x^{-1})|^2, \tilde{E}(s,.)\rangle+\lim_{s\to 0}\langle |\phi_0(.x^{-1})|^2,\Eis(F_s)\rangle.
\end{align*}
We call the first summand as the \emph{regularized term} which, upon rigorous application of the heuristic above, could be proved to be of bounded size. The second summand is called the \emph{degenerate term} and yields the main term. 

Up to some non-archimedean factors involving $L(1,\pi_0,\Ad)$ the degenerate term is of the form
$$\partial_{s=0}Z_x(f_{1/2}f_{1/2+s},W_0,\overline{W_0})-\partial_{s=0}Z_x(Mf_{1/2}Mf_{1/2+s},W_0,\overline{W_0}),$$
where $M$ is certain intertwining operator that arises in the constant term of a maximal parabolic Eisenstein series.
One main difficulty of the paper is asymptotically evaluating the above two derivatives. The first one is comparatively easy to understand as one can apply the support condition of $f_{1/2}f_{1/2+s}\left[\begin{pmatrix}I&\\c&1\end{pmatrix}x\right]$, which is concentrated on $c=O(1/X)$ and approximate invariance of $W_0$. The second one is more technical to analyze. The intertwined vector $Mf_{1/2+s}\left[\begin{pmatrix}I&\\c&1\end{pmatrix}x\right]$, which on the matrices $\begin{pmatrix}I&\\c&1\end{pmatrix}$ essentially mimics a Fourier transform of $f_{1/2+s}$, has support of size $c=O(X)$. So we cannot get away just with the invariance properties of $W_0$. In this case we understand a more detailed shape of the intertwined vectors via the Iwasawa decomposition on the matrices of the form $\begin{pmatrix}I&\\cX&1\end{pmatrix}$. This analysis is done in \S\ref{analysis-degenerate-period}.

On the other hand, to analyse the regularized term we understand the growth of the (degenerate) Fourier terms of $\Eis(f_s)$ for $s$ being close $0,1/2$ or $1$. This analysis relies on the analytic properties of the intertwining operators attached to various Weyl elements and functional analytic properties of the Eisenstein series. This analysis is done in \S\ref{analysis-regularized-period}.

\begin{remark}
We remark that our method of proof, which is uniform for $n\ge 3$ can also be made to work for $n=2$ with a slight modification with a modified main term (the statement of our theorem does not anyway make sense for $n=2$). The main terms in the asymptotic expansion are the artefacts of the non-integrable terms among the product of the constant terms in the Fourier expansion of $\Eis(f_{1/2})$ and $\Eis(f_{1/2+s})$. The constant term of $\Eis(f_s)$ looks like $\sum_w M_wf_s$ where $M_w$ are certain intertwining operators and $w$ runs over a set of Weyl elements attached to the underlying parabolic subgroup (see \S\ref{fourier-eisenstein-series}). If $n\ge 3$ then the non-integrable terms in the above mentioned product are of the form $f_{1/2}f_{1/2+s}$ and $Mf_{1/2}Mf_{1/2+s}$ where $M$ is the intertwiner attached to the relative long Weyl element. In particular, the off-diagonal terms of the from $f_{1/2} Mf_{1/2+s}$ are integrable. Such a phenomena does not happen for $n=2$. In this case (where the maximal Eisenstein series is also a minimal Eisenstein series), the off diagonal terms are also non-integrable. 

As described in the sketch of the proof, eventually we need to deform the principal series vector to regularize the Eisenstein series. The number of deformations needed in the Langlands parameters of the associated principal series vector depends on the number of non-integrable terms in the product of the constant terms. For $n \ge 3$ we need to deform only one of the parameters of the principal series vector to regularize the corresponding maximal parabolic Eisenstein series. However, for $n=2$ for the reasons stated above, to regularize the Eisenstein series we need to deform two (i.e.\ both) of the parameters. This modification would produce more degenerate terms, and consequently, a different main term with a different constant will appear; see \cite{BH}. 
\end{remark}

\subsection{What's next?}
As we have described above, the motivating question for us is to find an asymptotic expansion of the $2n$-th moment of the central $L$-values for $\PGL(n)$, and to do that we need to replace $\phi_0$ by a minimal Eisenstein series $E_0$ with trivial Langlands parameters. As, in particular, $E_0$ is not in $L^\infty$, our current proof obviously fails (see the sketch for the proof), and that is why we need to regularize $E_0$ as well. However, this regularization increases the analytic difficulties many fold. We need to employ a regularized version of the spectral decomposition (and Parseval), as in, e.g.\ \cite{N1, MV} to follow the same strategy as in the sketch for the proof of the main theorem. On the other hand, regularizing both the Eisenstein series involved in the period $\langle |E_0|^2,|\Eis(f_{1/2})(.x)|^2\rangle$ will introduce many more degenerate terms, which will typically have higher order poles at the critical point. This would likely yield a higher power of $\log X$ in the main term. It will be interesting to see if the constant appearing in the main term are the same as predicted by the random matrix models; see \cite{CFKRS}. However, we leave this to future work.

It is natural to speculate what happens for the second moment of the Rankin--Selberg $L$-functions for other $(\GL(n),\GL(m))$ pairs with $m\neq n$ and the $\GL(m)$ form being fixed (cuspidal or Eisenstein). If $m<n$, we believe that the problems become simpler than the $m=n$ case as the degrees, hence conductors, become lower. Similarly, for $m>n$ we expect the problems to be much more difficult for high degree and conductor size. In particular, it would be very interesting to see if we can push the method in this paper at least to the case $m=n+1$ case. More interestingly, if $n=3$ and the fixed form is a minimal Eisenstein series then we would have a Lindel\"of-consistent eighth moment (the convexity barrier) of $L$-functions of $\PGL(3)$ over the family $\F_X$.

\begin{remark}\label{explicate-constant-term}
We briefly remark that one may try to explicate the constant contribution of the asymptotic expansion in the main theorem and obtain a power saving error term as in \cite{BH}. One possible way to obtain finer asymptotics in the regularized part is to spectrally expand the period $\langle |\phi_0|^2,\tilde{E}_s(.x)\rangle$ over the $\PGL(n)$ automorphic spectrum. Then one may use the existence of a spectral gap and explicit decay of the matrix coefficient for $n\ge 3$ to obtain that $\langle |\phi_0|^2,\phi\rangle\langle\phi,\tilde{E}_s(.x)\rangle$, at least for a tempered $\phi$, will decay polynomially in $X$. However, it is not yet clear to us how to explicate the constant term and get an error term with polynomial saving in the degenerate part; see Remark \ref{explicate-constant-term-local}.
\end{remark}

\begin{ackno}
We generously thank Paul Nelson who is the author's doctoral supervisor for suggesting this problem and enormous guidance, feedback, and encouragement along the way. We also thank Valentin Blomer, Naser Sardari, Ramon Nunes, Gergely Harcos, Farrell Brumley, and Djordje Mili\'cevi\'c for their interest in this work and many useful feedbacks on an earlier draft of this paper. We thank ETH Z\"urich where the work has been started while the author was a doctoral student there. We also thank Max Planck Institute for Mathematics for providing a perfect research environment to complete the project. Finally, we are extremely grateful to the anonymous referees for their valuable time and comments.
\end{ackno}

\section{Basic Notations and Preliminaries}

\subsection{Basic notations}
We use ad\`elic language. Let $r\ge 3$. For any ring $R$ by $G(R)$ we denote the set of points $\GL_r(R)/R^\times$. In this paper $R$ will denote the ad\`eles $\A$ over $\Q$ or the local fields $\R,\Q_p$ or rational numbers $\Q$ or the local ring $\Z_p$. We drop the ring $R$ from the notation $G(R)$ if the ring is clear from the context.

Let $N$ be the maximal unipotent subgroup of $G$ consisting of upper triangular matrices. For $q\in\A^{r-1}$ we define a character of $N(\A)$ by
$$\psi_q(n(x))=\psi_0\left(\sum_{i=1}^{r-1}q_ix_{i,i+1}\right),\quad n(x):=(x_{i,j})_{i,j}.$$
where $\psi_0$ is an additive character of $\Q\backslash\A$. We abbreviate $\psi_{(1,\dots,1)}$ by $\psi$. We call $\psi_q$ \emph{non-degenerate} if $q_i\neq 0$ for $1\le i\le r-1$; otherwise, we call $\psi_q$ \emph{degenerate}.

Let $A$ be the set of diagonal matrices in $G$, which we identify with $\begin{pmatrix}A_{r-1}&\\&1\end{pmatrix}$ where $A_{r-1}$ is the set of diagonal matrices in $\GL(r-1)$. We parametrize elements of $A_{r-1}$ as $a(y):=\diag(y_1\dots y_{r-1},\dots, y_{r-1})$ Let $K:=\prod_{p\le \infty}K_p$ be the standard maximal compact in $G(\A)$ where $K_p:=G(\Z_p)$ for $p<\infty$ and $K_\infty:=\mathrm{PO}_r(\R)$.

For any factorizable function $f$ on $G(\A)$ by $f_p$ we denote the $p$-th component of $f$, which is a function on $G(\Q_p)$.

\subsection{Domains and measures}
We fix Haar measures on $G$ and its subgroups, and a $G$-invariant measure of $N\backslash G$. If the subgroup is compact then we normalize the Haar measure to be a probability measure. Let $\delta$ denote the modular character on $A$. It is defined by $$\delta(a(y)) := \prod_{j=1}^{r-1} |y_j|^{j(r-1-j)},$$ and is trivially extended to $NA$.

To integrate over $N(\R)\backslash G(\R)$ we use two different types of coordinates according to efficiency.
The first one is Bruhat (with respect to the standard maximal parabolic) coordinates.
First, note that the set of elements of the form $\begin{pmatrix}h&b^t\\c&0\end{pmatrix}$ with $h\in \GL_{r-1}(\R)$ and row vectors $b,c\in\R^{r-1}$, has zero measure in $\GL_r(\R)$ with respect to its Haar measure. Thus while integrating over $G(\R)$ we integrate over the points of the form
$$\begin{pmatrix}h&b^t\\&1\end{pmatrix}\begin{pmatrix}\mathrm{I}_{r-1}&\\c&1\end{pmatrix},\quad h\in \GL_{r-1}(\R),\text{ and } b,c\in\R^{r-1} \text{ row vectors}.$$
Similarly, the set of points of the form
$$g = \begin{pmatrix}h&\\&1\end{pmatrix}\begin{pmatrix}\mathrm{I}_{r-1}&\\c&1\end{pmatrix},\quad h\in N_{r-1}(\R)\backslash\GL_{r-1}(\R),\text{ and } c\in\R^{r-1}\text{ row vector}$$
has full measure in $N(\R)\backslash G(\R)$.
We 
use these coordinates to integrate over $N(\R)\backslash G(\R)$ using the invariant measure
$$dg=\frac{dh}{|\det(h)|}dc,$$
where $dc$ denotes the Lebesgue measure and $dh$ is the $\GL_{r-1}(\R)$-invariant Haar measure on $N_{r-1}(\R)\backslash\GL_{r-1}(\R)$. Description of the above invariant measure follows from \cite[eq. (5.14)]{Kn} and discussion above that. However, there is a more direct way to see this. Let $\phi\in C_c(G(\R))$ measurable. Then, it follows from \cite[Proposition 1.4.3]{Go} that
$$\int_{G(\R)} \phi(g) dg = \int_{\R^{n^2-1}}\phi\left[\begin{pmatrix}h&b^t\\c&1\end{pmatrix}\right]\frac{\prod_{i,j}d_L h_{ij}\prod_id_Lb_i\prod_i d_Lc_i}{\left|\det\left[\begin{pmatrix}h&b^t\\c&1\end{pmatrix}\right]\right|^n},$$
where $d_Lx$ denotes the Lebesgue measure on $\R$. Noting that
$$\begin{pmatrix}h&b^t\\c&1\end{pmatrix}=\begin{pmatrix}\mathrm{I}_{r-1}&b^t\\&1\end{pmatrix}\begin{pmatrix}h-b^tc\\&1\end{pmatrix}\begin{pmatrix}\mathrm{I}_{r-1}&\\c&1\end{pmatrix}$$
we can write
$$\int_{G(\R)} \phi(g) dg = \int_{\R^{n^2-1}}\phi\left[\begin{pmatrix}\mathrm{I}_{r-1}&b^t\\&1\end{pmatrix}\begin{pmatrix}h-b^tc\\&1\end{pmatrix}\begin{pmatrix}\mathrm{I}_{r-1}&\\c&1\end{pmatrix}\right]\frac{\prod_{i,j}d_L h_{ij}\prod_id_Lb_i\prod_id_Lc_i}{|\det(h-b^tc)|^n}.$$
Fixing $b,c$ and changing variables $h_{ij}\mapsto h_{ij}+(b^tc)_{ij}$ we can write the above as
$$\int_{G(\R)} \phi(g) dg = \int_{\R^{n^2-1}}\phi\left[\begin{pmatrix}\mathrm{I}_{r-1}&b^t\\&1\end{pmatrix}\begin{pmatrix}h\\&1\end{pmatrix}\begin{pmatrix}\mathrm{I}_{r-1}&\\c&1\end{pmatrix}\right]\frac{\prod_{i,j}d_L h_{ij}\prod_id_Lb_i\prod_id_Lc_i}{|\det(h)|^n}.$$
Noting that $\frac{\prod_{i,j}d_Lh_{ij}}{|\det(h)|^{n-1}}=dh$ and taking $N(\R)$-quotient on the left we deduce the invariant measure on $N(\R)\backslash G(\R)$.

On the other hand when we integrate on $N_{r-1}(\R)\backslash\GL_{r-1}(\R)$ we use Iwasawa coordinates. We write
$$N_{r-1}(\R)\backslash\GL_{r-1}(\R)\ni h=a(y)k,\quad a(y)\in A_{r-1}, k\in K_{r-1},$$
where $K_{r-1}$ is the standard maximal compact in $\GL_{r-1}$,
with the measure
$$dh=\frac{\prod_id^\times y_i}{\delta(a(y))}dk,$$
where $dk$ is the probability Haar measure on $K_{r-1}$.

Let $\X:=G(\Q)\backslash G(\A)$.
We fix a fundamental domain $\X$ in $G(\A)$ of the form 
$$D\times K_f,\quad D\subseteq G(\R), K_f:=\prod_{p<\infty}K_p$$
that is contained in a Siegel domain of the form $\mathbb{S}\times K_f$ where
\begin{equation}\label{siegel-domain}
    \mathbb{S}:=\{G(\R)\ni g=n(x)\begin{pmatrix}a(y)&\\&1\end{pmatrix}k\mid |x_{i,j}|<1,|y_i|> y_0, k\in K_\infty\},
\end{equation}
where $y_0>0$ is an explicit constant dependent only on the group. The above follows from strong approximation for $\GL(n)$ and \cite[\S1.3]{Go}.

We equip $\X$ with the $G(\A)$-invariant probability measure that in Iwasawa coordinates is given by
$$\X\ni g=n(x)\begin{pmatrix}a(y)&\\&1\end{pmatrix}k,\quad dg=\prod_{j,k}dx_{j,k}\frac{\prod_id^\times y_i}{\delta(a(y))|\det(a(y))|}dk,$$
where $n(x)\in N(\R)$ and $dx_{i,j}$ is the usual Lebesgue measure. Note that $\delta\begin{pmatrix}a(y)&\\&1\end{pmatrix}=\delta(a(y))|\det(a(y))|$.

\subsection{Automorphic representations}\label{automorphic-representation}
We briefly describe the classes of local and global representations that are relevant in this paper. We refer to \cite{MW}, \cite[\S5]{CPS2} for details.

Let $\hat{\X}$ be the isomorphism class of irreducible unitary automorphic representations that are unramified at all finite places and appear in the spectral decomposition of $L^2(\X)$. Similarly, by $\hat{\X}_\gen$ we denote the subclass of \emph{generic} representations in $\hat{\X}$ i.e.\ the class of representations that have (unique) Whittaker models.

We first mention the Langlands description for $\hat{\X}_\gen$. We take a partition $r=r_1+\dots+r_k$. Let $\pi_j$ be a unitary cuspidal automorphic representation for $\GL_{r_j}(\Q)$ (if $r_j=1$ we take $\pi_j$ to be a unitary Hecke character). Consider the normalized parabolic induction $\Pi$ from the Levi $\GL(r_1)\times\dots\times\GL(r_k)$ to $G$ of the tensor product $\pi_1\otimes\dots\otimes\pi_k$. There exists a unique irreducible constituent of $\Pi$, which we denote by the isobaric sum $\pi_1\boxplus\dots\boxplus\pi_k$. Then the Langlands classification says that every element in $\hat{\X}_\gen$ is isomorphic to an isobaric sum $\boxplus_{j=1}^{k'}\pi'_j$ for some partition $r=\sum_{j=1}^{k'}r'_j$ and some cuspidal representation $\pi'_j$ of $\GL(r'_j)$.

We recall from \cite{MW} that we call $\pi'\in \hat{\X}$ a \emph{discrete series} if $\pi'$ appears discretely in the spectral decomposition of $L^2(\X)$. The elements of $\pi'$ are square-integrable automorphic forms for $G(\Q)$. M\oe glin--Waldspurger classified the discrete series for $G(\Q)$ via the iterated residues of generic automorphic forms; see \cite{MW}. Langlands description of $\hat{\X}$ says that every element in $\hat{\X}$ is isomorphic to an isobaric sum $\boxplus_{j=1}^{k'}\pi'_j$ for some partition $r=\sum_{j=1}^{k'}r'_j$ and some discrete series representation $\pi'_j$ of $\GL(r'_j)$.

We fix an automorphic Plancherel measure $d\mu_\aut$ on $\hat{\X}$ compatible with the invariant probability measure on $\X$. If $\pi$ is a discrete series then $d\mu_\aut(\pi)$ is absolutely continuous to the counting measure at $\pi$. On the other hand, if $\pi$ is an Eisenstein series induced from a twisted discrete series $\pi'|.|^\lambda$ where $\pi'$ is a discrete series on a Levi subgroup $M$ and $\lambda$ lies in the purely imaginary dual of the Cartan subalgebra of $M$ then $d\mu_\aut(\pi)$ is absolute continuous to the product of the counting measure at $\pi'$ and the Lebesgue measure at $\lambda$.

For any $\pi\in \hat{\X}$ we denote the $p$-th component of $\pi$ by $\pi_p$ for $p\le \infty$. The generalized Ramanujan conjecture predicts that if $\pi$ is cuspidal then $\pi_p$ is tempered for all $p\le \infty$. In this paper we assume that certain cuspidal representations are \emph{$\vartheta$-tempered} at the archimedean place, whose definition we recall below.

First we describe the Langlands description of the generic representations of $G(\R)$. Let $r'\in\{1,2\}$ and $\sigma$ be an \emph{essentially} square integrable (square integrable mod center) representation of $\GL_{r'}(\R)$. That is, if $r'=2$ then $\sigma$ is a discrete series of $\GL_2(\R)$ and if $r'=1$ then $\sigma$ is a unitary character of $\GL_1(\R)$. 
By the Langlands classification, we 
know that any \emph{generic} unitary irreducible representation $\xi$ of $G(\R)$ is isomorphic to a normalized parabolic induction of
$$\sigma_1|\det|^{s_1}\otimes\dots\otimes\sigma_k|\det|^{s_k},$$
from a Levi subgroup attached to a partition of $r=\sum_{j=1}^k r_j$ with $r_j\in\{1,2\}$ where $\sigma_j$ is an essentially square integrable representation of $\GL(r_j)$, and $s_j\in\C$ with $\sum_{j=1}^k r_js_j=0$ and $\Re(s_1)\ge\dots\ge\Re(s_k)$.

Let $\vartheta \geq 0$. We say that $\xi$ is \textit{$\vartheta$-tempered} if all such $s_i$ have real parts in $[-\vartheta,\vartheta]$. By \cite{MS} if $\pi$ is cuspidal then $\pi_\infty$ is $\vartheta$-tempered with $\vartheta = 1/2 - 1/(1+r^2)$.

We denote the analytic conductor of $\pi$ by $C(\pi)$. Note that as $\pi\in\hat{\X}$ is unramified at all the finite places we have $C(\pi)=C(\pi_\infty)$. If $\{\mu_i\}\in\C^r$ are the Langlands parameters of $\pi_\infty$ then we define (see \cite{IS})
$C(\pi_\infty):=\prod_{i=1}^r(1+|\mu_i|)$.

\subsection{Maximal Eisenstein Series}\label{eisenstein-series}
Let $P$ be the standard parabolic subgroup in $G$ attached to the $r=(r-1)+1$ partition.
We choose a generalized principal series vector
\begin{equation*}
    f_s\in \I_{r-1,1}(s):=\mathrm{Ind}_{P(\A)}^{G(\A)}|\det|^{s}\boxplus |.|^{-(r-1)s},\quad s\in \C,
\end{equation*}
by
\begin{equation}\label{relation-f-phi}
    f_s(g)=f_{s,\Phi}(g):=\int_{\A^\times}\Phi(te_rg)|\det(tg)|^sd^\times t.
\end{equation}
where $\Phi\in \mathcal{S}(\A^r)$ is a Schwartz--Bruhat, factorizable function and $e_r=(0,\dots,0,1)\in\A^r$. The integral in \eqref{relation-f-phi} converges for $\Re(s)>1/r$ and then can be extended meromorphically to the whole complex plane. By $\hat{\Phi}$ we denote Fourier transform of $\Phi$ that is defined by
$$\hat{\Phi}(x):=\int_{\A^r}\Phi(u)\psi_0(x_1u_1+\dots++x_ru_r)du.$$
We abbreviate $f_{s,\hat{\Phi}}$ as $\hat{f}_s$.
We record the following transformation property of $f_s$, which can be seen from \eqref{relation-f-phi},
\begin{equation}\label{transformation-f}
    f_s\left[\begin{pmatrix}h&\\&1\end{pmatrix}g\right]=|\det(h)|^sf_s(g),\quad h\in \GL_{r-1}(\A),g\in G(\A).
\end{equation}
Finally, we define the maximal Eisenstein series associated to $f_s$ by
\begin{equation*}
    \mathrm{Eis}(f_{s})(g):=\sum_{\gamma\in P(\Q)\backslash G(\Q)}f_s(\gamma g).
\end{equation*}
The above definition is valid for $s$ in a right half plane and then can be extended to all of $\C$ by meromorphic continuation. From \cite[\S4]{JS}, \cite[\S2.3.1]{Co} we know that for $f_{s,\Phi}\in \I_{r-1,1}(s)$ with some $\Phi\in\mathcal{S}(\A^r)$ the maximal parabolic Eisenstein series $\Eis(f_{s,\Phi})$ has at most simple poles at $s=0$ and $s=1$. The residues at these poles are independent of $g$.


Let $\tilde{P}$ be the maximal parabolic subgroup in $G$ attached to the partition $r=1+(r-1)$ (the \emph{associate parabolic} to $P$). We can similarly construct an associated Eisenstein series from a vector $\tilde{f}_s\in \I_{1,r-1}(s)$ defined analogously. All of the properties of an Eisenstein series associated to $P$ hold analogously for the same associated to $\tilde{P}$. 

\subsection{Genericity and Kirillov model}\label{whittaker-kirillov-model}

We briefly review the Whittaker and Kirillov models of a generic representation of $G$ over a local field; see \cite{J2} for details. In this subsection we only work locally, without mentioning the underlying local field. Fix a non-degenerate additive character $\psi$ of $N<G$. Consider the space of Whittaker functions on $G$ by
$$\W(G):=\left\lbrace W\in C^\infty(G)\middle| \begin{aligned}
& W(ng)=\psi(n)W(g),n\in N, g\in G;\\& W \text{ grows at most polynomially in }g
\end{aligned}\right\rbrace,$$
on which $G$ acts by right translation.

We call an irreducible representation $\pi$ of $G$ \emph{generic} if there exists a $G$-equivariant embedding $\pi\hookrightarrow \W(G)$. For generic $\pi$ we identify $\pi$ with its image in $\W(G)$, which we call the \emph{Whittaker model} of $\pi$ under this embedding.

It is known that (e.g.\ see \cite{J2} for the case of an archimedean local field) from the theory of the Kirillov model that if $\pi$ is an irreducible generic representation of $\PGL(r)$ then
$$\pi\ni W\mapsto \left\lbrace\GL(r-1)\ni g\mapsto W\left[\begin{pmatrix}g&\\&1\end{pmatrix}\right]\right\rbrace$$
is injective and the space of the restricted Whittaker functions in the right hand side, which is called the \emph{Kirillov model}, is isomorphic to $\pi$ as well. It is also known that the space $C_c^\infty(N_{r-1}(\R)\backslash\GL_{r-1}(\R),\psi)$ is contained in $\pi$ under this realization; see \cite[Proposition 5]{J2}.

If $\pi$ is also unitary then we can give a unitary structure on its Whittaker model by the inner product
$$\langle W_1,W_2\rangle :=\int_{N_{r-1}(\R)\backslash\GL_{r-1}(\R)}W_1\left[\begin{pmatrix}h&\\&1\end{pmatrix}\right]\overline{W_2\left[\begin{pmatrix}h&\\&1\end{pmatrix}\right]}dh,$$
i.e.\ we have $\langle W_1(.g),W_2(.g)\rangle=\langle W_1,W_2\rangle$ for $g\in G$.

\subsection{Zeta integrals}\label{zeta-integrals}
We review the theory of the $\GL(r)\times\GL(r)$ Rankin--Selberg integral. We refer to \cite{Co} for details. 
We choose $\phi_0\in\pi_0$ with a factorizable Whittaker function $W_{0}$ and a maximal Eisenstein series $\Eis(f_s)$ attached to some vector $f_s$ in the generalized Principal series $\I_{r-1,1}(s)$, as defined in \S\ref{eisenstein-series}.
Let $\pi\in\hat{\X}$ and $\phi\in\pi$ be any automorphic form. One defines the global Rankin--Selberg integral \cite[\S 2.3.2]{Co} of $\phi_0$ and $\phi$ by
\begin{equation*}
    \Psi(f_s,\phi_0,\bar{\phi}):=\int_{G(\Q)\backslash G(\A)}\phi_0(g)\overline{\phi(g)}\Eis(f_s)(g)dg,
\end{equation*}
where $s\in\C$ is such that $\Eis(f_s)$ is regular.
As $\phi_0$ is cuspidal the above integral converges absolutely. For $s$ in a right half plane performing a standard unfolding-folding one gets
$$\Psi(f_s,\phi_0,\bar{\phi})=\int_{P(\Q)\backslash G(\A)}\phi_0(g)\overline{\phi(g)}f_s(g)dg.$$
The above integral representation of $\Psi(f_s,\phi_0,\bar{\phi})$ has a meromorphic continuation to all $s\in \C$. It is known that if $\pi$ and $\pi_0$ are cuspidal then the only possible poles of $\Psi$ are simple and can occur at $\Re(s)=0,1$.

We may choose $f_s$ to be factorizable, which can be done by choosing $\Phi\in\mathcal{S}(\A^r)$, as in \S\ref{eisenstein-series}, to be factorizable. Furthermore, if we assume that $\pi$ is generic and $\phi\in\pi$ has a factorizable Whittaker function $W_\phi$ then for all $s\in\C$ the global zeta integral is \emph{Eulerian} i.e.\ factors in local zeta integrals:
\begin{equation*}
    \Psi(f_s,\phi,\bar{\phi})=\Psi_\infty(f_{s,\infty},W_{0,\infty},\overline{W_{\phi,\infty}})\prod_{p<\infty}\Psi_p(f_{s,p},W_{0,p},\overline{W_{\phi,p}}),
\end{equation*}
where the local zeta integral $\Psi_\infty$ is defined by
\begin{align*}
    \Psi_\infty(f_{s,\infty},W_{0,\infty}, \overline{W_{\phi,\infty}}):&=\int_{N(\R)\backslash G(\R)}W_{0,\infty}(g)\overline{W_{\phi,\infty}(g)}f_{s,\infty}(g) dg\\
    &=\int_{N(\R)\backslash \GL_r(\R)}W_{0,\infty}(g)\overline{W_{\phi,\infty}(g)}\Phi_{\infty}(e_rg)|\det(g)|^sdg,
\end{align*}
for $s$ being in some right half plane and then can be meromorphically continued to the whole complex plane. Similarly, the non-archimedean zeta integral $\Psi_p$ is defined by replacing $\infty$ with $p$ and $\R$ with $\Q_p$. It is known that if $\pi_{0,p}$ and $\pi_p$ are unitary and $\vartheta_0$ and $\vartheta$ tempered, respectively, then the above integral representation of $\Psi_p$ is valid for $\Re(s)\ge 1/2$ if $\vartheta+\vartheta_0 < 1/2$ and $p\le \infty$ (this can be seen in the archimedean case from the bounds of the Whittaker functions in Lemma \ref{bound-for-convergence}).

We record the local functional equation satisfied by $\Psi_\infty$. From \cite[Theorem 3.2]{Co} we have
\begin{multline*}
    \int_{N(\R)\backslash \GL_r(\R)}\tilde{W}_{0,\infty}(g){\tilde{W}_{\phi,\infty}(g)}\hat{\Phi}_{\infty}(e_rg)|\det(g)|^{1-s}dg\\=\gamma_\infty(s,\pi_{0,\infty}\otimes{\pi}_\infty,\psi)\int_{N(\R)\backslash \GL_r(\R)}W_{0,\infty}(g){W_{\phi,\infty}(g)}\Phi_{\infty}(e_rg)|\det(g)|^sdg.
\end{multline*}
Here $\tilde{W}$ denotes the contragredient Whittaker function of  $W$ defined by $\tilde{W}(g):=W(wg^{-t})$ where $w$ is the long Weyl element in $G(\R)$ and $\gamma_\infty(.,.,\psi)$ denotes the local archimedean $\gamma$-factor. As the additive character $\psi$ is fixed through out the paper we drop $\psi$ from the notation of $\gamma_\infty$. Folding the above integrals over $\R^\times$ we can also rewrite the local functional equation as
\begin{equation}\label{local-functional-equation}
    \Psi_\infty(\hat{f}_{1-s,\infty}, \tilde{W}_{0,\infty},\overline{\tilde{W}_{\infty}})=\gamma_\infty(s,\pi_{0,\infty}\otimes\bar{\pi}_{\infty})\Psi_\infty(f_{s,\infty},W_{0,\infty},\overline{W_{\infty}}),
\end{equation}
for any $W_\infty\in \pi_\infty$ and $f$ is related to $\Phi$ according to \eqref{relation-f-phi}. From the definition of the $\gamma$-factors (see \cite[p. 120]{Co}) one can check that $|\gamma_\infty(1/2,\Pi)|=1$ if $\Pi$ is unitary.

\subsection{Plancherel formula}
We refer to \cite[\S2.2]{MV} for a more detailed discussion of the Plancherel formula. 

Recall the automorphic Plancherel measure $d\mu_\aut$ on $\hat{\X}$ from \S\ref{automorphic-representation}.
Let $\phi_1,\phi_2\in C^\infty(\X)$ with rapid decay at all cusps. We record a Plancherel formula (i.e.\ a spectral decomposition) of the inner product between $\phi_1$ and $\phi_2$;
\begin{equation}\label{spectral-decomposition}
    \langle \phi_1,\phi_2\rangle = \int_{\hat{\X}}\sum_{\phi\in\B(\pi)}\langle \phi_1,\phi\rangle\langle \phi,\phi_2\rangle d\mu_\aut(\pi),
\end{equation}
where $\B(\pi)$ is an orthonormal basis of $\pi$ and
$$\langle f_1,f_2\rangle:=\int_{\X}f_1(g)\overline{f_2(g)}dg.$$
The identity \eqref{spectral-decomposition} is independent of choice of $\B(\pi)$.

Rapid decay properties of $\phi_i$ imply that all the inner products on the right hand side of \eqref{spectral-decomposition} converges. One can show by the trace class property in $L^2(\X)$ of some inverse Laplacian that the right hand side of \eqref{spectral-decomposition} converges absolutely.

\subsection{Spectral weights}
Let $\pi,\pi_0\in\hat{\X}_\gen$ with $W_{0,\infty}\in\pi_{0,\infty}$ and $f_s\in \I_{r-1,1}(s)$. We define the spectral weight
\begin{equation}\label{spectral-weight}
J({f_{s,\infty}W_{0,\infty}},\pi_\infty):=\sum_{W_\infty\in\B(\pi_\infty)}|\Psi_\infty(f_{s,\infty},W_{0,\infty}, \overline{W_\infty})|^2,
\end{equation}
here $\B(\pi_\infty)$ is an orthonormal basis of $\pi_\infty$. The sum in the right hand side of \eqref{spectral-weight} is absolutely convergent and is independent of choice of $\B(\pi_\infty)$; see \cite[Appendix 4]{BM}. 


The definition of $J$ involves only the archimedean components of the representations and functions. In fact, one can define $J$ for any irreducible generic unitary representation $\sigma$ of $G(\R)$ and $\beta\in C^\infty(N(\R)\backslash G(\R),\psi_\infty)$ with sufficient decay at infinity, by
\begin{equation*}
    J(\beta,\sigma):=\sum_{W\in\B(\sigma)}\left|\int_{N(\R)\backslash G(\R)}\beta(g)\overline{W(g)}dg\right|^2.
\end{equation*}
Then using the Whittaker--Plancherel formula (see \cite[Chapter 15]{W}) one can obtain that
\begin{equation}\label{whittaker-plancherel}
    \int_{\widehat{G(\R)}}J(\beta,\sigma)d\mu_{\mathrm{loc}}(\sigma)=\int_{N(\R)\backslash G(\R)}|\beta(g)|^2dg=:\|\beta\|^2_{L^2(N(\R)\backslash G(\R))}.
\end{equation}
Here $\widehat{G(\R)}$ is the tempered unitary dual of $G(\R)$ equipped with the local Plancherel measure $d\mu_{\mathrm{loc}}$ compatible with the chosen Haar measure on $G(\R)$ (see \cite[\S4.13.2]{BrMi}).

\subsection{Analytic newvectors}\label{analytic-newvector}
Analytic newvectors are certain approximate archimedean analogues of the classical non-archimedean newvectors pioneered by Casselman \cite{Ca} and Jacquet--Piatetski-Shapiro--Shalika \cite{JPSS1}. Let $K_0(p^N)\subset \PGL_r(\Z_p)$ be the subgroup of matrices whose last rows are congruent to $(0,\dots,0,*)\mod p^N$. Let $\sigma$ be a generic irreducible representation of $\PGL_r(\Q_p)$ and let $N_0$ be the minimal non-negative integer such that $\sigma$ contains a non-zero vector $v$ that is invariant by $K_0(p^{N_0})$. Let $C(\sigma)$ be the conductor of $\sigma$, which can be defined in terms of the local gamma factor attached to $\sigma$. Then the main theorem of \cite{Ca,JPSS1} states that the real number $p^{N_0}$ is equal to $C(\sigma)$. One calls such a $v$ a \emph{newvector} of $\sigma$. In \cite{JPSS1}, the authors called newvectors as \emph{essential vectors} and in some literature the authors called them as \emph{newforms}.

In \cite{JN} the authors produce an approximate analogue of this theorem at the archimedean place. Let $X>1$ be tending to infinity and $\tau>0$ be sufficiently small but fixed. We define an \emph{approximate congruence subgroup} $K_0(X,\tau)\subseteq\PGL_r(\R)$, which is an archimedean analogue of the subgroup $K_0(p^N)$, in the following way: it is the image in $\PGL_r(\R)$ of
\begin{equation}\label{defn-congruence-subgroup}
    \left\{\begin{pmatrix}
    a&b\\c&d
    \end{pmatrix}\in\mathrm{GL}_{r}(\R)\middle|
    \begin{aligned}
    & a \in \GL_{r-1}(\R),\quad
    |a - 1_{r-1}| < \tau,
    \quad
    |b|<\tau, \\
    & d \in \GL_1(\R),\quad
    |c|<\frac{\tau}{X},
    \quad
    |d-1|<\tau
    \end{aligned}
    \right\}.
\end{equation}
Here, $|.|$ denotes an arbitrary fixed
norm on the corresponding spaces of matrices. Fix $0\le\vartheta<1/2$. Then in \cite[Theorem 1]{JN} the authors show that for all $\epsilon>0$ there is a $\tau>0$ such that for all generic irreducible unitary $\vartheta$-tempered representation $\pi$ of $\PGL_r(\R)$ there is a unit vector $v\in \pi$ such that
$$\|\pi(g)v-v\|_\pi<\epsilon\quad \text{ for all } g\in K_0(C(\pi),\tau),$$
where $C(\pi)$ is the analytic conductor of $\pi$. We call such a vector $v$ an \emph{analytic newvector} of $\pi$.

The authors also prove that \cite[Theorem 7]{JN} any unit vector $v$ that in the Kirillov model of $\pi$ can be given by a function in $C_c^\infty(N_{r-1}(\R)\backslash \GL_{r-1}(\R),\psi_\infty)^{\mathrm{O}_{r-1}(\R)}$ is a newvector. Moreover, $v$ can be chosen in a way such that if $W$ is the image of $v$ in the corresponding Whittaker model then also
$$|W(g)-W(1)|<\epsilon$$
for all $g\in K_0(C(\pi),\tau)$ and $W(1)\asymp 1$.

\subsection{Main theorem}
\begin{theorem}\label{main-theorem}
Let $r\ge 3$ and $X$ be tending to infinity.
Let $\pi_0$ be a fixed cuspidal representation in $\hat{\X}$ such that $\pi_{0,\infty}$ is $\vartheta_0$-tempered for some $0\le \vartheta_0< 1/(r^2+1)$. We define a weight function
$$J_X:\hat{\X}_\gen\to\R_{\ge 0},$$ as in \eqref{normalized-spectral-weight}, which satisfies the following properties:
\begin{itemize}
    \item $J_X(\pi)$ only depends on the archimedean component of $\pi$ (with an abuse of notation we write $J_X(\pi)=J_X(\pi_\infty)$).
    \item If $\pi_\infty$ is $\vartheta$-tempered such that $\vartheta+\vartheta_0<1/2$ and $C(\pi_\infty)<X$ then $J_X(\pi_\infty)\gg_{\pi_0} 1$.
    \item $\int_{\widehat{G(\R)}}J_X(\pi_\infty)d\mu_\loc(\pi_\infty)=X^{r-1}$.
\end{itemize}
And finally, we have
$$\int_{\hat{\X}_\gen}\frac{|L(1/2,\tilde{\pi}\otimes \pi_0)|^2}{\ell(\pi)}{J_X(\pi)}d\mu_\aut(\pi)= X^{r-1}\left(r\frac{\zeta(r/2)^2}{\zeta(r)}L(1,\pi_0,\Ad)\log X + O_{\pi_0}(1)\right),$$
where $\tilde{\pi}$ is the contragredient of $\pi$.
Here $\ell(\pi)$ is defined as in \eqref{defn-ell-pi} and only depends on the non-archimedean data of $\pi$.
\end{theorem}
If $\pi$ is cuspidal then $\ell(\pi)\asymp L(1,\pi,\Ad)$ with an absolute implied constant and thus $\ell(\pi) \ll_\epsilon C(\pi)^\epsilon$, which follows from \cite{Li}. 

We note that if $\pi\in\hat{\X}_\gen$ then $\pi_\infty$ is $\vartheta$-tempered for $\vartheta<1/2-1/(r^2+1)$, which is a result in \cite{MS}. Thus the $\vartheta_0$-temperedness assumption of $\pi_{0,\infty}$ in Theorem \ref{main-theorem} implies that $J_X(\pi)\gg 1$ for all $\pi\in\hat{\X}_\gen$ with $C(\pi)<X$. Moreover, the family
$$\F_X^\gen:=\{\text{generic automorphic representations $\pi$ of $\PGL_r(\Z)$ with $C(\pi)<X$}\}$$
has $\ell(\pi)^{-1}$-weighted cardinality $\asymp X^{r-1}$. This is essentially contained in the proof of \cite[Theorem 9]{JN}. Hence, $J_X$ can be realized as a smoothened characteristic function of the $\ell(\pi)^{-1}$-weighted family $\F_X^\gen$.

Consequently, we have an immediate corollary of Theorem \ref{main-theorem}.

\begin{cor}\label{corollary}
Let $\pi_0$ be as in Theorem \ref{main-theorem}. Then
$$\sum_{\substack{{C(\pi)<X}\\{\hat{\X}\ni\pi\textrm{ cuspidal}}}}\frac{|L(1/2,\pi\otimes \pi_0)|^2}{L(1,\pi,\Ad)}\ll_{\pi_0} X^{r-1}\log X,$$
as $X$ tends to infinity.
\end{cor}
This is the sharpest possible (Lindel\"of on average) second moment estimate of the cuspidal Rankin--Selberg central $L$-values.

\section{The Fourier Expansion of Maximal Eisenstein Series}\label{fourier-eisenstein-series}
We recall some useful information about the Fourier expansion of maximal Eisenstein series. The computation is essentially done in \cite{La}, however, we extract the relevant computation for completeness.

Let $f_s$ be a holomorphic section in the generalized principal series $\I_{r-1,1}(s)$ such that $f_s$ is constructed from a Schwartz--Bruhat function $\Phi\in\mathcal{S}(\A^r)$, as described in \S\ref{eisenstein-series}. Let $\Eis(f_s)$ be the Eisenstein series attached to $f_s$.

We want to understand the Fourier expansion of $\Eis(f_s)$. It is a straightforward calculation using the Bruhat decomposition. We sketch out the essential details for completeness. Let $\Re(s)$ be sufficiently large. We temporarily allow $\psi$ to be a possibly degenerate character of $N$. Then
\begin{equation}\label{sum-integral}
    \int_{N(\Q)\backslash N(\A)}\Eis(f_s)(ng)\overline{\psi(n)}dn=\sum_{\gamma\in P(\Q)\backslash G(\Q)}\int_{N(\Q)\backslash N(\A)}f_s(\gamma ng)\overline{\psi(n)}dn.
\end{equation}

We start by the Bruhat decomposition of $G(\Q)$ with respect to $P(\Q)$. The Bruhat cells are indexed by a subset of the Weyl group, namely, the subset of Weyl elements $w$ such that $w\alpha>0$ for all simple roots $\alpha$ other than $\alpha_0$ that determines $P$.

\begin{lemma}\label{bruhat-decomposition}
We define the Weyl elements
\begin{equation*}
    w_i:=\begin{pmatrix}I_{i-1}&&\\&&I_{r-i}\\&1&\end{pmatrix},\quad 1\le i\le r.
\end{equation*}
Also let $N_i$ be the subgroup of $N$ of the form
\begin{equation*}
    N_i:=\left\lbrace n:=\begin{pmatrix}I_{i-1}&&\\&1&x\\&&I_{r-i}\end{pmatrix}\mid x:=(x_1,\dots, x_{r-i})\right\rbrace.
\end{equation*}
Then
$$G(\Q)=\bigsqcup_{i=1}^r P(\Q)w_iN_i(\Q).$$
where the union is disjoint.
\end{lemma}

\begin{proof}
Any $\gamma\in G(\Q)$ has the bottom row of the form $(0,\dots,0,d,*,\dots,*)$ where $d\neq 0$ and occurs at the $i$-th position for some $1\le i\le r$. There exists an element $x\in N_i$ such that
$\gamma=d\gamma'x$ with $\gamma'$ having bottom row of the form $(0,\dots,0,1,0,\dots,0)$ with $1$ at the $i$-th position. It can readily be checked that $\gamma'w_i^{-1}\in P(\Q)$. Clearly, the union is disjoint.
\end{proof}

Using Lemma \ref{bruhat-decomposition} and the left-$P(\Q)$ invariance of $f_s$ we can rewrite the right hand side of \eqref{sum-integral} as
\begin{equation*}
    \sum_{i=1}^r\sum_{\gamma\in N_i(\Q)}\int_{N(\Q)\backslash N(\A)}f_s(w_i\gamma ng)\overline{\psi(n)}dn.
\end{equation*}
Note that $N_i:=N\cap w_i^{-1}N^tw_i$. Hence, $N=N_i\bar{N}_i$ where $\bar{N}_i:=N\cap w_i^{-1}N w_i$. It can be checked that
$$\bar{N}_i=\{n\in N\mid e_in=e_i\},$$
where $e_i=(0,\dots,0,1,0,\dots,0)$ with $1$ at the $i$-th place.

We write an element $n\in N$ as $n_1n_2$ with $n_1\in N_i$ and $n_2\in\bar{N}_i$. Unfolding the $N_i(\Q)$ sum we obtain that the right hand side of \eqref{sum-integral} is equal to
\begin{equation*}
    \sum_{i=1}^r\int_{\bar{N}_i(\Q)\backslash \bar{N}_i(\A)}\overline{\psi(n_2)}\int_{N_i(\A)}f_s(w_in_1n_2g)\overline{\psi(n_1)}dn_1dn_2.
\end{equation*}
There exists $n_2'\in\bar{N}_i$ and $n_1'\in N_i$ such that $n_2'n_1'=n_1n_2$ and 
and $n_2''\in N$ such that $n_2''w_i=w_in_2'$. Appealing to the left-$N(\A)$ invariance of $f_s$ we conclude that the above expression is 
\begin{equation*}
    \sum_{i=1}^r\int_{\bar{N}_i(\Q)\backslash \bar{N}_i(\A)}\overline{\psi(n_2)}\int_{N_i(\A)}f_s(w_in_1'g)\overline{\psi(n_1)}dn_1dn_2.
\end{equation*}
We check that if $e_in_1=(0,1,x)$ for some $x\in\A^{r-i}$ then $n_1'=(0,1,xu)$ for some upper triangular unipotent matrix $u$ in $\GL_{r-i}(\A)$. Also $\psi(n_1)$ is equal to $\psi(n_1')$. Thus making the change of variables $xu\mapsto x$, we obtain that 
\begin{equation}\label{expansion-weyl-terms}
    \int_{N(\Q)\backslash N(\A)}\Eis(f_s)(ng)\overline{\psi(n)}dn=\sum_{i=1}^r\int_{\bar{N}_i(\Q)\backslash \bar{N}_i(\A)}\overline{\psi(n')}dn'\int_{N_i(\A)}f_s(w_ing)\overline{\psi(n)}dn.
\end{equation}
Clearly, if $\psi$ is non-degenerate the above is zero. In particular, if $\psi$ is of the form $\psi_{\tilde{q}}$ for some $\tilde{q}:=(q_j)_j\in\Q^{r-1}$ then the $i$-th summand, for $i<r$, on the right hand side of \eqref{expansion-weyl-terms} does not identically vanish only if $q_j=0$ for all $j\neq i$. For $i=r$ the same happens only if $\tilde{q}=0$ in which case the summand is equal to $f_s(g)$. For $q\in\Q$ we denote $(0,\dots,0,q,0,\dots,0)$, where $q$ is at the $i$-th place, by $i(q)$.

We define (again on a right half plane, and extend by meromorphic continuation) twisted intertwining operators on $\I_{r-1,1}(s)\ni f_s$ attached to the Weyl element $w_i$ by
\begin{equation}\label{defn-intertwiner}
    M_i^qf_s(g):=\int_{N_i(\A)} f_s(wng)\overline{\psi_{i(q)}(n)}dn.
\end{equation}
Thus we obtain the following Fourier expansion of $\Eis(f_s)(g)$.

\begin{lemma}\label{fourier-expansion-basic}
Let $f_s$ and $\Eis(f_s)(g)$ as above. Then
$$\Eis(f_s)(g)=f_s(g)+\sum_{i=1}^{r-1}\sum_{q\in \Q}M_i^qf_s(g).$$
The terms $f_s$ and $M_i^0f_s$ are the constant terms of $\Eis(f_s)(g)$.
\end{lemma}

Let us write $g\in G(\A)$ in its Iwasawa coordinates $g=nak$ where $a:=\begin{pmatrix}a(y)&\\&1\end{pmatrix}$. Then for $i<r$ we have
\begin{align*}
    M_{i}^qf_s(g)
    &=\int_{N_{i}(\A)}f_s(w_in'nak)\overline{\psi_{i(q)}(n)}dn',
\end{align*}
which is defined for $\Re(s)$ large enough and can be meromorphically continued.

We work exactly as before to compute the above integral. 
We write $n=n_1n_2$ with $n_1\in N_i$ and $n_2\in\bar{N}_i$
and make the change of variables $n'\mapsto n'n_1^{-1}$. Then we write $w_in'n_2=n_2'w_in''$ for some $n_2'\in N$ and $n''$ that is related to $n'$ as before and we make the change of variables $n''\mapsto n'$. We use the left $N(\A)$-invariance of $f_s$ and the fact that $\psi(n')=\psi(n'')$.

Finally, we make the change of variables $n'\mapsto an'a^{-1}$ and use transformation property of $f_s$ as in \eqref{transformation-f} to obtain
\begin{equation}\label{y-asymptotic-fourier-coeff}
    M_{i}^qf_s(g)=\psi_{i(q)}(n_1)\prod_{j=1}^{i-1}|y_j|^{sj}\prod_{j=i}^{r-1}|y_j|^{(1-s)(r-j)}\int_{N_i(\A)}f_s(w_ink)\overline{\psi_{i(q)}(ana^{-1})}dn.
\end{equation}

We first study the integral on the right hand side of \eqref{y-asymptotic-fourier-coeff} for $q=0$. We use the construction of $f_s$ using $\Phi\in\mathcal{S}(\A^r)$ as in \S\ref{eisenstein-series}. We also parametrize $n$ so that $e_in=(0,1,x)$ with $x\in \A^{r-i}$ and make the change of variables $x\mapsto x/t$ to write the integral as
\begin{equation*}
    \int_{\A^{r-i}}\int_{\A^\times}(k.\Phi)(0,t,x)|t|^{rs-r+i}d^\times tdx.
\end{equation*}
Here $(k.\Phi)(x):=\Phi(xk)$.
Using Tate's functional equation (see \cite[Proposition 3.1.6]{Bum}) we can rewrite the above as
\begin{equation*}
    \int_{\A^\times}\widehat{(k.\Phi)}^i(te_i)|t|^{r-i+1-rs}d^\times t,
\end{equation*}
where the partial Fourier transform $\hat{\Phi}^i$ is defined by
\begin{equation*}
    \hat{\Phi}^i(x_1,\dots,x_r):=\int_{\A^{r-i+1}}\Phi(x_1,...,x_{i-1}, u_1,\dots,u_{r-i+1})\psi(x_iu_1+\dots+x_ru_{r-i+1})du.
\end{equation*}
In particular, it can be seen that 
\begin{equation}\label{intertwined-vector}
    M_{1}f_s(g)=f_{1-s,\hat{\Phi}}(wg^{-t})=:\tilde{f}_s(g).
\end{equation}
where $w$ is the long Weyl element. It can be checked that $\tilde{f}$ lies in the principal series $\I_{1,r-1}(1-s)$ arising from the associate parabolic $\tilde{P}$.

Now for $q\neq 0$ the integral on the right hand side of \eqref{y-asymptotic-fourier-coeff} gives rise to a degenerate Whittaker function. Parametrizing $n\in N_i(\A)$ as in Lemma \ref{bruhat-decomposition} one can see that
$\psi_{i(q)}(ana^{-1})=\psi_0(qy_ix_1)$, i.e.\ the value $\psi_{i(q)}(ana^{-1})$ depends on $a$ only through $y_i$.
We define
\begin{equation*}
    W^i_{f_s}(qy_i,k):=\int_{N_i(\A)}f_s(w_ink)\overline{\psi_{i(q)}(ana^{-1})}dn,
\end{equation*}
Again, the above is defined for $\Re(s)$ sufficiently large and can be extended analytically to all of $\C$ and shown that $W^i_{f_s}(t,k)$ decays rapidly as $t\to \infty$. We prove these claims in Lemma \ref{bound-compact-part} (although these results are implicitly done in \cite{J3}).
In particular, we have 
$$M^q_if_s(g)=W^i_{f_s}(qy_i,k)\psi_0(qx_{i,i+1}),$$
for $q\neq 0$.

We summarize the above results and re-write Lemma \ref{fourier-expansion-basic} in the following proposition to record the Fourier expansion of a maximal Eisenstein series.
\begin{prop}\label{fourier-expansion-eisenstein-series}
Let $f_s\in \I_{r-1,1}(s)$ be a holomorphic section and $\Eis(f_s)(g)$ be the corresponding maximal Eisenstein series. Let $g=n(x)\begin{pmatrix}a(y)&\\&1\end{pmatrix}k$ be its Iwasawa decomposition. Then
$$\Eis(f_s)(g)=f_s(g)+\sum_{i=1}^{r-1}\prod_{j=1}^{i-1}|y_j|^{sj}\prod_{j=i}^{r-1}|y_j|^{(1-s)(r-j)}\left[M_i^0f_s(k)+\sum_{q\in\Q^\times}W^i_{f_s}(qy_i,k)\psi_0(qx_{i,i+1})\right].$$
The terms containing $M_i^0$ are the constant terms of $\Eis(f_s)(g)$ and the terms containing $W^i_{f_s}$ are holomorphic in $s$. The above sum converges absolutely and uniformly on compacta.
\end{prop}

\section{Proof of the Main Theorem}
\subsection{Choices of the local components}\label{local-choice}
We start by choosing various vectors and auxiliary test functions. Let $\pi_0\in\hat{\X}_\gen$ be the fixed cuspidal representation as in Theorem \ref{main-theorem}. Let $\phi_0\in\pi_0$ with Whittaker function $W_0=\bigotimes_{p\le \infty}W_{0,p}$, such that 
$$W_{0,p}\text{ are unramified for $p<\infty$ with } W_{0,p}(1)=1.$$ 
Here and elsewhere in the paper the index set $\{p\le \infty\}$ denotes the set of places $\{\infty\}\cup\{p\text{ prime in }\Z\}$.

We choose $W_{0,\infty}\in \pi_{0,\infty}$ so that $$\|W_{0,\infty}\|_{\pi_{0,\infty}}=1,\text{ and }W_{0,\infty}\left[\begin{pmatrix}.&\\&1\end{pmatrix}\right]\in C_c^\infty(N_{r-1}(\R)\backslash\GL_{r-1}(\R),\psi_\infty)^{\mathrm{O}_{r-1}(\R)},$$
whose existence is guaranteed by the theory of Kirillov model.
We choose $\mathcal{S}(\A^r)\ni\Phi=\bigotimes_{p\le \infty}\Phi_p$ with $$\Phi_p:=\mathrm{char}({\Z_p^r})=\hat{\Phi}_p\text{ for }p<\infty,$$ and for $\tau>0$ sufficiently small but fixed
$$\Phi_\infty\in C_c^\infty(B_\tau(0,...,0,1)),$$
such that $\Phi_\infty$ is non-negative and has values sufficiently concentrated near $1$. Here $B_\tau$ denotes the ball of radius $\tau$. Thus $\Phi_\infty$ can be thought as a smoothened characteristic function of $B_\tau(0,...,0,1)$.

Let $f_s:=f_{s,\Phi}\in \I_{r-1,1}(s)$ be associated to $\Phi$ according to \eqref{relation-f-phi}. The choice of $\Phi_\infty$ ensures that there exist sufficiently small $\tau_1,\tau_2>0$ (depending on $\tau$ and $\Phi_\infty$) such that
$f_{1/2,\infty}\left[\begin{pmatrix}\mathrm{I}_{r-1}&\\c&1\end{pmatrix}\right]$ is supported on $|c|\le\tau_1$ and has values $\asymp 1$ for $|c|\le \tau_2$. This implies that
$$\int_{\R^{r-1}}\left|f_{1/2,\infty}\right|^2\left[\begin{pmatrix}\mathrm{I}_{r-1}&\\c&1\end{pmatrix}\right]dc\asymp 1,$$
with an absolute implied constant.
We re-normalize $\Phi_\infty$ by an absolute constant so that the above integral is $1$.

\subsection{Computation of the spectral side}
Let $\Eis(f_s):=\Eis(f_{s,\Phi})$ be the maximal Eisenstein series attached to $f_{s,\Phi}\in \I_{r-1,1}(s)$, which is defined in \S\ref{local-choice}. Let $X>1$ be a large number tending to infinity and 
$$A(\A)\ni x:=(x_p)_p,\quad x_\infty:=\diag(X,\dots,X,1)\in A(\R)\text{ and }x_p=1 \text{ for all } p<\infty.$$

Our point of departure is the following period, which we write in two different ways:
\begin{equation}\label{main-period}
    \int_{\X}|\phi_0(g)|^2|\Eis(f_s)(gx)|^2dg=\langle \phi_0 \Eis(f_s)(.x),\phi_0\Eis(f_s)(.x)\rangle.
\end{equation}
We use the Parseval relation \eqref{spectral-decomposition} and the notations in \S\ref{zeta-integrals} to write the right hand side of \eqref{main-period} as
\begin{equation}\label{spectral-decomposition-main-period}
    \int_{\hat{\X}}\sum_{\phi\in\B(\pi)}|\Psi(f_s(.x),\phi_0,\bar{\phi})|^2d\mu_\aut(\pi),
\end{equation}
where $\B(\pi)$ is an orthonormal basis of $\pi$.

\begin{lemma}\label{non-generic-zero}
Let $\pi\in\hat{\X}\setminus\hat{\X}_\gen$ be a non-generic representation. Then $\Psi(f_s,\phi_0,\bar{\phi})=0$ for all $\phi\in \pi$ and $s\in \C$.
\end{lemma}

\begin{proof}
For $\Re(s)$ sufficiently large we have (see \S\ref{zeta-integrals})
$$\Psi(f_s,\phi_0,\bar{\phi})=\int_{P(\Q)\backslash G(\A)}\phi_0(g)\overline{\phi(g)}f_s(g)dg.$$
We follow the computation of \cite[p.104-105]{Co}. We use the Fourier expansion
$$\phi_0(g)=\sum_{\gamma\in N(\Q)\backslash P(\Q)}W_0(\gamma g),$$
the left $P(\Q)$-invariance of $f_s$,
and unfold over $P(\Q)$ to get
$$\Psi(f_s,\phi_0,\bar{\phi})=\int_{N(\Q)\backslash G(\A)}W_0(g)\overline{\phi(g)}f_s(g)dg.$$
We fold the last integral over $N(\A)$, use the left $N$-equivariance of $W_0$, and the left $N$-invariance of $f_s$ to obtain
$$\Psi(f_s,\phi_0,\bar{\phi})=\int_{N(\A)\backslash G(\A)}W_0(g)f_s(g)\int_{N(\Q)\backslash N(\A)}\overline{\phi(ng)}\psi(n)dn dg.$$
By definition, the inner integral vanishes as $\phi$ is non-generic. Finally, by analytic continuation of $\Psi$ we extend the result for all $s\in \C$.
\end{proof}

Thus Lemma \ref{non-generic-zero} allows us to reduce the integral in \eqref{spectral-decomposition-main-period} only over $\hat{\X}_\gen$. Once we restrict to $\pi\in\hat{\X}_\gen$ then we can use the Eulerian property of the zeta integral $\Psi$ as in \S\ref{zeta-integrals}. If $\phi\in\pi$ with $\|\phi\|_\pi=1$ and Whittaker function $W_\phi=\bigotimes_{p\le \infty}W_{p}$ such that $W_{p}$ is unramified and $W_{p}(1)=1$ for $p<\infty$ then by Schur's lemma we have
\begin{equation}\label{defn-ell-pi}
    \|\phi\|^2_\pi=\ell(\pi)\|W_\infty\|^2_{\pi_\infty},
\end{equation}
where $\ell(\pi)$ only depends on the non-archimedean data of $\pi$.
A standard Rankin--Selberg computation\footnote{See \cite[eq. (3.11)]{B3} and the computation there for the spherical case; the general case follows similarly.} yields that $\ell(\pi)\asymp L(1,\pi,\Ad)$ for a cuspidal $\pi$. In fact, in our case $\ell(\pi)$ is equal to $L(1,\pi,\Ad)$ up to a positive constant dependent only on $n$.

Another standard computation \cite[Theorem 3.3]{Co} shows that
\begin{equation}\label{unramified-zeta-integral}
    \Psi_p(f_{s,p},W_{0,p},\overline{W_p})=L_p(s,\pi_0\otimes\bar{\pi}),\quad p<\infty,
\end{equation}
where $f_s$ is as chosen in \S\ref{local-choice} and $L_p(s,.)$ denotes the unramified $p$-adic Euler factor of $L(s,.)$.
Thus by meromorphic continuation we have
\begin{equation*}
    \Psi(f_s,\phi_0,\bar{\phi})={L(s,\pi_0\otimes\bar{\pi})}\Psi_\infty(f_{s,\infty},W_{0,\infty},\overline{W_\infty}),
\end{equation*}
for all $s\in\C$ whenever both sides of the above are defined. Using the equation above, and recalling the harmonic weight in \eqref{defn-ell-pi} and the spectral weight in \eqref{spectral-weight} we obtain that \eqref{spectral-decomposition-main-period} is equal to
\begin{equation*}
    \int_{\hat{\X}_\gen}\frac{|L(s,\pi_0\otimes\bar{\pi})|^2}{\ell(\pi)}J(f_{s,\infty}(.x_\infty) W_{0,\infty},\pi_\infty)d\mu_\aut(\pi).
\end{equation*}
We appeal to the holomorphicity of the zeta integrals to specify $s=1/2$ and define a normalized spectral weight $J_X(\pi_\infty)$ by
\begin{equation}\label{normalized-spectral-weight}
    J_X(\pi):=J_X(\pi_\infty):=X^{r-1}J(f_{1/2,\infty}(.x_\infty) W_{0,\infty},\pi_\infty).
\end{equation}
Thus we write the main equation of our proof:
\begin{equation}\label{period-spectral-side}
    \int_{\hat{\X}_\gen}\frac{|L(1/2,\pi_0\otimes\bar{\pi})|^2}{\ell(\pi)}J_X(\pi)d\mu_\aut(\pi)=X^{r-1}\langle |\phi_0|^2, |\Eis(f_{1/2})(.x)|^2\rangle.
\end{equation}

\subsection{Computation of the period side}
Again recall the choices of the local factors in \S\ref{local-choice}. We write $f_{s}=\bigotimes_{p\le \infty}f_{s,p}$; then for $k_p\in K_p$ 
\begin{align*}
    f_{s,p}(k_p)=\int_{\Q_p^\times}\Phi_p(te_rk_p)|\det(tk_p)|^sd^\times t.
\end{align*}
Here we fix Haar measures $d^\times t$ on $\Q_p^\times$ and (resp.\ $dt$ on  $\Q_p$) such that $\mathrm{vol}(\Z_p^\times)=1$ (resp.\ $\mathrm{vol}(\Z_p)=1$).

First, we record that $f_{s,p}$ is an unramified vector in $\I_{r-1,1}(s)_p$ and $\tilde{f}_{s,p}$
is an unramified vector in $\I_{1,r-1}(1-s)_p$, which is a generalized principal series attached to the opposite parabolic of $P$. Note that
$te_rk_p\in \Z_p^r$ if and only if $t\in \Z_p$. Thus $$f_{s,p}(k_p)=\sum_{m=0}^\infty p^{-mrs}=(1-p^{-rs})^{-1},$$
for $\Re(s)>0$.
Similarly, using \eqref{intertwined-vector} we have $$\tilde{f}_{s,p}(k_p)=f_{1-s,\hat{\Phi}_p}(wk_p^{-t})=(1-p^{-r(1-s)})^{-1},$$
for $\Re(s)<1$. Thus for $g$ in the fundamental domain of $\X$ we can write
\begin{equation}
    f_s(g)=\zeta(rs)f_{s,\infty}(g_\infty)
\end{equation}
and
\begin{equation}
    \tilde{f}_{s}(g)=\zeta(r-rs)\tilde{f}_{s,\infty}(g_\infty),
\end{equation}
for all $s\in\C$, which can be achieved by meromorphic continuation.

Let $\Re(s)$ be sufficiently small. From Proposition \ref{fourier-expansion-eisenstein-series} we can see that among the constant terms of $\Eis(f_{1/2+s})(g)$ the terms that do not lie in $L^2(\X)$ are $f_{1/2+s}(g)$ and $\tilde{f}_{1/2+s}(g)$. Similarly, one checks that the constant terms of
$$\overline{\Eis(f_{1/2})}\Eis(f_{1/2+s})-\overline{f_{1/2}}f_{1/2+s}-\overline{\tilde{f}_{1/2}}\tilde{f}_{1/2+s}$$
are integrable in $L^2(\X)$.
Inspired by this we define a regularized Eisenstein series of the form
\begin{equation}\label{regularized-eisenstein-series}
\tilde{E}_s:=\overline{\Eis(f_{1/2})}\Eis(f_{1/2+s})-\Eis(\overline{f_{1/2}}f_{1/2+s})-\Eis(\overline{\tilde{f}_{1/2}}\tilde{f}_{1/2+s}).
\end{equation}

\begin{proof}[Proof of Theorem \ref{main-theorem}]
Recall \eqref{period-spectral-side} and \eqref{regularized-eisenstein-series}. We write the inner product on the right hand side of \eqref{period-spectral-side} as 
$$\lim_{s\to 0}\langle |\phi_0|^2,\tilde{E}_s(.x)\rangle+\lim_{s\to 0}\left[\langle |\phi_0|^2, \Eis(\overline{f_{1/2}}f_{1/2+s})(.x)\rangle+\langle |\phi_0|^2,\Eis(\overline{\tilde{f}_{1/2}}\tilde{f}_{1/2+s})(.x)\rangle\right].$$
The second term is the degenerate term as in \eqref{degenerate-terms}. From \eqref{main-term}, Proposition \ref{technical-main-term}, and Lemma \ref{archimedean-factors} we obtain that the second term above is 
$$rL(1,\pi_0,\Ad)\frac{\zeta(r/2)^2}{\zeta(r)}\log X + O_{\pi_0}(1).$$
On the other hand, we write the first term above, which is the regularized term as
$$\lim_{s\to 0}\int_{\X}|\phi_0|^2(gx^{-1})\overline{\tilde{E}_s(g)}dg,$$
and bound this by
$$\|\phi_0\|^2_{L^\infty(\X)}\int_{\X}|\tilde{E}_s(g)|dg.$$
From Proposition \ref{regularized-period} we know that the last integral is convergent for $s$ being sufficiently small, and $\tilde{E}_s$ is holomorphic in a sufficiently small neighbourhood of $s=0$. Thus using Cauchy's residue theorem we can write the above limit as
$$\int_{|s|=\epsilon}\frac{1}{s}\int_{\X}|\phi_0|^2(gx^{-1})\overline{\tilde{E}_s(g)}dg\frac{ds}{2\pi i},$$
for some arbitrary small but fixed $\epsilon>0$.
Applying Proposition \ref{regularized-period} once again we confirm the above integral is $O_{\phi_0,\epsilon}(1)$. 

Now non-negativity and the first property of the spectral weight $J_X(\pi)$ follow from the definition \eqref{normalized-spectral-weight}. The second property follows from Proposition \ref{main-prop-spectral-side}. Finally, the third property follows from \eqref{whittaker-plancherel} and Lemma \ref{archimedean-factors}.
\end{proof}

\section{Analysis of the Degenerate Terms in the Period Side}\label{analysis-degenerate-period}
In this section we analyse the degenerate terms
\begin{equation}\label{degenerate-terms}
\lim_{s\to 0}\left[\langle |\phi_0|^2, \Eis(\overline{f_{1/2}}f_{1/2+s})(.x)\rangle+\langle |\phi_0|^2,\Eis(\overline{\tilde{f}_{1/2}}\tilde{f}_{1/2+s})(.x)\rangle\right].
\end{equation}
Note that $\overline{f_{1/2}}f_{1/2+s}\in \I_{r-1,1}(1+s)$ is such that its local component $\overline{f_{1/2,p}}f_{1/2+s,p}$ is unramified for $p<\infty$. Thus by the uniqueness of spherical vector, $\overline{f_{1/2,p}}f_{1/2+s,p}\in \I_{r-1,1}(1+s)_p$ is a multiple of the unramified vector
$$g\mapsto\int_{\Q_p^\times}\Phi_p(te_rg)|\det(tg)|^{1+s}d^\times t.$$
Comparing the values of the functions at the identity as before we check that the multiple is 
$$\frac{(1-p^{-r/2})^{-1}(1-p^{-r/2-rs})^{-1}}{(1-p^{-r-rs})^{-1}}\text{ for }p<\infty.$$ 
We compute the first term inside the limit in \eqref{degenerate-terms} for $\Re(s)$ large. Doing a similar computation as in \S\ref{zeta-integrals} and using \eqref{unramified-zeta-integral} we obtain
\begin{multline}\label{degenerate-term1}
    \langle |\phi_0|^2, \Eis(\overline{f_{1/2}}f_{1/2+s})(.x)\rangle = \frac{\zeta(r/2)\zeta(r/2+rs)}{\zeta(r+rs)}L(1+s,\pi_0\otimes\tilde{\pi}_0)\\
    \Psi_\infty(\overline{f_{1/2,\infty}}f_{1/2+s,\infty}(.x_\infty),W_{0,\infty},\overline{W_{0,\infty}}).
\end{multline}
Finally, we meromorphically continue the above to the whole complex plane. 

Similarly, we compute the second term inside the limit in \eqref{degenerate-terms}. Note that in this case $\tilde{f}$ lies in $\I_{1,r-1}$ associated to the parabolic $\tilde{P}$. Working as in \S\ref{zeta-integrals} we obtain that
$$\langle |\phi_0|^2, \Eis(\overline{\tilde{f}_{1/2}}\tilde{f}_{1/2+s})(.x)\rangle=\int_{\tilde{P}(\Q)\backslash G(\A)}|\phi_0|^2(g)\overline{\tilde{f}_{1/2}}\tilde{f}_{1/2+s}(gx)dg.$$
We recall definition of $\tilde{f}$ in \eqref{intertwined-vector} and make the change of variables $g\mapsto wg^{-t}$ to obtain that the above is equal to
$$\int_{P(\Q)\backslash G(\A)}|\tilde{\phi}_0|^2(g)\overline{\hat{f}_{1/2}}\hat{f}_{1/2-s}(gx^{-1})dg,$$
where $\tilde{\phi}_0(g):=\phi_0(wg^{-t})$, which lies in the contragredient representation $\tilde{\pi}_0$. Note that $\hat{\Phi}_p=\Phi_p$ for $p<\infty$. Thus doing a similar calculation to that preceding \eqref{degenerate-term1} we obtain
\begin{multline}\label{degenerate-term2}
    \langle |\phi_0|^2, \Eis(\overline{\tilde{f}_{1/2}}\tilde{f}_{1/2+s})(.x)\rangle = \frac{\zeta(r/2)\zeta(r/2-rs)}{\zeta(r-rs)}L(1-s,\pi_0\otimes\tilde{\pi}_0)\\
    \Psi_\infty(\overline{\hat{f}_{1/2,\infty}}\hat{f}_{1/2-s,\infty}(.x_\infty^{-1}),\tilde{W}_{0,\infty},\overline{\tilde{W}_{0,\infty}}).
\end{multline}
Recalling the definition of the contragredient $\tilde{W}_0$ and making the change of variables $g_\infty\to wg_\infty^{-t}$ in the definition of the zeta integral $\Psi_\infty$ we also have
$$\Psi_\infty(\overline{\hat{f}_{1/2,\infty}}\hat{f}_{1/2-s,\infty}(.x_\infty^{-1}),\tilde{W}_{0,\infty},\overline{\tilde{W}_{0,\infty}})=\Psi_\infty(\overline{\tilde{f}_{1/2,\infty}}\tilde{f}_{1/2+s,\infty}(.x_\infty),W_{0,\infty},\overline{W_{0,\infty}}).$$
In the following Lemma \ref{archimedean-factors} we prove the archimedean factors $\Psi_\infty$ on the right hand side of \eqref{degenerate-term1} and \eqref{degenerate-term2} are equal for $s=0$. We first record that
$$\Psi_\infty(|h|^2, W_{0,\infty}, \overline{W_{0,\infty}})=\|hW_{0,\infty}\|^2_{L^2(N(\R)\backslash G(\R))},$$
where $h$ is either $f_{1/2,\infty}(.x_\infty)$ or $\tilde{f}_{1/2,\infty}(.x_\infty)$.

\begin{lemma}\label{archimedean-factors}
Recall the choices of the local components in \S\ref{local-choice}. We have
$$\|W_{0,\infty}f_{1/2,\infty}(.x_\infty)\|^2=\|W_{0,\infty}\tilde{f}_{1/2,\infty}(.x_\infty)\|^2=\|W_{0,\infty}f_{1/2,\infty}\|^2=1,$$
where all the norms are taken in ${L^2(N(\R)\backslash G(\R))}$.
\end{lemma}

\begin{proof}
To ease the notations we drop $\infty$ from the subscripts in this proof.

First recall that
$$\tilde{f}_{1/2}(g)=\hat{f}_{1/2}(wg^{-t})$$
which implies, by a change of variable $g\mapsto wg^{-t}$, that
$$\|W_{0}\tilde{f}_{1/2}(.x)\|^2=\int_{N(\R)\backslash G(\R)}|W_0(g)|^2|\hat{f}_{1/2}(wg^{-t}x^{-1})|^2dg=\int_{N(\R)\backslash G(\R)}|\tilde{W}_0(g)|^2|\hat{f}_{1/2}(gx^{-1})|^2dg.$$
We make the change of variables $g\mapsto gx$ and then employ the Whittaker--Plancherel formula as in \eqref{whittaker-plancherel} to write the above as
$$\|\tilde{\pi}_0(x)\tilde{W}_{0}\hat{f}_{1/2}\|^2
=\int_{\widehat{G(\R)}}\sum_{W\in \B(\sigma)}|\Psi(\hat{f}_{1/2},\tilde{\pi}_0(x)\tilde{W}_{0},{W})|^2d\mu_{\mathrm{loc}}(\sigma).$$
We use the $\GL(r)\times\GL(r)$ local functional equation as in \eqref{local-functional-equation} and the unitarity of the gamma factor at $1/2$ to obtain that
$$|\Psi(\hat{f}_{1/2},\tilde{\pi}_0(x)\tilde{W}_{0},{W})|^2=|\Psi({f}_{1/2},\pi_0(x^{-1})W_{0},\tilde{W})|^2.$$
Consequently, applying the Whittaker--Plancherel again with the orthonormal basis $\B(\sigma):=\{\sigma(x)\tilde{W}\}$ we obtain
$$\|{W}_{0}\tilde{f}_{1/2}(.x)\|^2=\|\tilde{\pi}_0(x)\tilde{W}_{0}\hat{f}_{1/2}\|^2=\|{\pi}_0(x^{-1}){W}_{0}{f}_{1/2}\|^2=\|W_{0}f_{1/2}(.x)\|^2,$$
which proves the first equality.

Thus now it is enough to prove that
$$\|W_{0}f_{1/2}\|^2=\|W_{0}f_{1/2}(.x)\|^2.$$
We use Bruhat coordinates to write
\begin{multline*}
    \int_{N(\R)\backslash G(\R)}|W_{0}|^2(g)|f_{1/2}|^2\left[g\begin{pmatrix}X\mathrm{I}_{r-1}&\\&1\end{pmatrix}\right]dg\\=\int_{N_{r-1}(\R)\backslash\GL_{r-1}(\R)}\int_{\R^{r-1}}|W_{0}|^2\left[\begin{pmatrix}h&\\c&1\end{pmatrix}\right]|f_{1/2}|^2\left[\begin{pmatrix}hX&\\cX&1\end{pmatrix}\right]dc\frac{dh}{|\det(h)|}.
\end{multline*}
Using the transformation property of $f_{1/2}$ as in \eqref{transformation-f} and making the change of variables $c\mapsto c/X$ we obtain the above is equal to
$$\int_{\R^{r-1}}|f_{1/2}|^2\left[\begin{pmatrix}\mathrm{I}_{r-1}&\\c&1\end{pmatrix}\right]\int_{N_{r-1}(\R)\backslash\GL_{r-1}(\R)}|W_{0}|^2\left[\begin{pmatrix}h&\\&1\end{pmatrix}\begin{pmatrix}\mathrm{I}_{r-1}&\\c/X&1\end{pmatrix}\right]dhdc.$$
Using the $G$-invariance of the inner product in the Whittaker model as in \S\ref{whittaker-kirillov-model} we conclude that the inner integral above is equal to
$$\int_{N_{r-1}(\R)\backslash\GL_{r-1}(\R)}|W_{0}|^2\left[\begin{pmatrix}h&\\&1\end{pmatrix}\right]dh=\int_{N_{r-1}(\R)\backslash\GL_{r-1}(\R)}|W_{0}|^2\left[\begin{pmatrix}h&\\&1\end{pmatrix}\begin{pmatrix}\mathrm{I}_{r-1}&\\c&1\end{pmatrix}\right]dh.$$
Thus reverse engineering the above manipulation with Bruhat coordinates (that is, taking $X=1$) we conclude the proof of the first two equalities. 

From the above proof we also obtain that
$$\|W_{0}f_{1/2}\|^2=\|W_{0}\|^2_{\pi_0}\int_{\R^{r-1}}\left|f_{1/2}\right|^2\left[\begin{pmatrix}\mathrm{I}_{r-1}&\\c&1\end{pmatrix}\right]dc.$$
We deduce the last equality recalling the normalizations of $W_{0}$ and $f_{1/2}$.
\end{proof}

It is known that (see \cite[Theorem 4.2]{Co}) if $\pi_0$ is cuspidal then $L(s,\pi_0\otimes\tilde{\pi}_0)$ has a simple pole at $s=1$ with residue $L(1,\pi_0,\Ad)$. Let us write
$$L(1+s,\pi_0\otimes\tilde{\pi}_0)=\frac{L(1,\pi_0,\Ad)}{s}+O_{\pi_0}(1),$$
as $s\to 0$. Thus using \eqref{degenerate-term1}, \eqref{degenerate-term2}, and Lemma \ref{archimedean-factors} we can evaluate the limit in \eqref{degenerate-terms} as
\begin{multline}\label{main-term}
    \lim_{s\to 0}\left[\langle |\phi_0|^2, \Eis(\overline{f_{1/2}}f_{1/2+s})(.x)\rangle+\langle |\phi_0|^2,\Eis(\overline{\tilde{f}_{1/2}}\tilde{f}_{1/2+s})(.x)\rangle\right]=\\
    L(1,\pi_0,\Ad)\frac{\zeta(r/2)^2}{\zeta(r)}\Psi'(f_{1/2,\infty},W_{0,\infty})+O_{\pi_0,\Phi}(1),
\end{multline}
where $\Psi'(f_{1/2,\infty},W_{0,\infty})$ is defined as
$$\partial_{s=0}\left(\Psi_\infty(\overline{f_{1/2,\infty}}f_{1/2+s,\infty}(.x_\infty),W_{0,\infty},\overline{W_{0,\infty}})- \Psi_\infty(\overline{\tilde{f}_{1/2,\infty}}\tilde{f}_{1/2+s,\infty}(.x_\infty),W_{0,\infty},\overline{W_{0,\infty}})\right).$$
Here and elsewhere in the paper we write $\partial_{s=0}$ as an abbreviation of $\left.\frac{\partial}{\partial s}\right\vert_{s=0}$.

\begin{prop}\label{technical-main-term}
We have 
$$\Psi'(f_{1/2,\infty},W_{0,\infty})=r\log X+O_{W_{0,\infty},\Phi_\infty}(1),$$
as $X$ tends to infinity.
\end{prop}

Proposition \ref{technical-main-term} follows immediately from the following Lemma \ref{dual-term}, Lemma \ref{easy-term}, and Lemma \ref{archimedean-factors}. Again to ease the notations we drop $\infty$ subscripts from the proofs of the next two lemmata.

\begin{lemma}\label{dual-term}
We have
\begin{equation*}
    \partial_{s=0}\Psi_\infty(\overline{\tilde{f}_{1/2,\infty}}\tilde{f}_{1/2+s,\infty}(.x_\infty),W_{0,\infty},\overline{W_{0,\infty}})=-\log X\|W_{0,\infty}f_{1/2,\infty}\|^2+O_{W_{0,\infty},\Phi_\infty}(1),
\end{equation*}
as $X$ tends to infinity.
\end{lemma}

\begin{proof}
We start by make the change of variables $g\mapsto wg^{-t}$ in the zeta integral to write
$$\Psi(\overline{\tilde{f}_{1/2}}\tilde{f}_{1/2+s}(.x),W_0,\overline{W_0})=\Psi(\overline{\hat{f}_{1/2}}\hat{f}_{1/2-s}(.x^{-1}),\tilde{W}_0,\overline{\tilde{W}_0}).$$ 
We use Bruhat coordinates as in the proof of Lemma \ref{archimedean-factors} to write the above zeta integral as
$$\int_{N_{r-1}(\R)\backslash\GL_{r-1}(\R)}\int_{\R^{r-1}}|\tilde{W}_{0}|^2\left[\begin{pmatrix}h&\\c&1\end{pmatrix}\right]\overline{\hat{f}_{1/2}}\hat{f}_{1/2-s}\left[\begin{pmatrix}h/X&\\c/X&1\end{pmatrix}\right]dc\frac{dh}{|\det(h)|}.$$
Again as in the proof of Lemma \ref{archimedean-factors} we use the transformation property of $\overline{\hat{f}_{1/2}}\hat{f}_{1/2-s}$ as in \eqref{transformation-f} and make the change of variables $c\mapsto cX$ to obtain the above is equal to 
$$X^{(r-1)s}\int_{\R^{r-1}}\overline{\hat{f}_{1/2}}\hat{f}_{1/2-s}\left[\begin{pmatrix}\mathrm{I}_{r-1}&\\c&1\end{pmatrix}\right]\int_{N_{r-1}(\R)\backslash\GL_{r-1}(\R)}|\tilde{W}_{0}|^2\left[\begin{pmatrix}h&\\cX&1\end{pmatrix}\right]|\det(h)|^{-s}dhdc.$$
Differentiating at $s=0$ we obtain the above is equal to
\begin{multline}\label{s-derivative}
    (r-1)\log X\int_{\R^{r-1}}|\hat{f}_{1/2}|^2\left[\begin{pmatrix}\mathrm{I}_{r-1}&\\c&1\end{pmatrix}\right]\int_{N_{r-1}(\R)\backslash\GL_{r-1}(\R)}|\tilde{W}_{0}|^2\left[\begin{pmatrix}h&\\cX&1\end{pmatrix}\right]dhdc\\
    +\int_{\R^{r-1}}\overline{\hat{f}_{1/2}}\partial_{s=0}\hat{f}_{1/2-s}\left[\begin{pmatrix}\mathrm{I}_{r-1}&\\c&1\end{pmatrix}\right]\int_{N_{r-1}(\R)\backslash\GL_{r-1}(\R)}|\tilde{W}_{0}|^2\left[\begin{pmatrix}h&\\cX&1\end{pmatrix}\right]dhdc\\
    -\int_{\R^{r-1}}|\hat{f}_{1/2}|^2\left[\begin{pmatrix}\mathrm{I}_{r-1}&\\c&1\end{pmatrix}\right]\int_{N_{r-1}(\R)\backslash\GL_{r-1}(\R)}|\tilde{W}_{0}|^2\left[\begin{pmatrix}h&\\cX&1\end{pmatrix}\right]\log|\det(h)|dhdc.
\end{multline}

The first summand in \eqref{s-derivative} is easy to understand. Using the invariance of the unitary product exactly as in the proof of Lemma \ref{archimedean-factors} we can yield the first summand is equal to
$$(r-1)\log X\|\tilde{W}_0\|_{\tilde{\pi}_0}\|f_{1/2}\|^2.$$
From the Whittaker--Plancherel expansion \eqref{whittaker-plancherel}, the $\GL(r)\times\GL(r-1)$ local functional equation (see \cite[Proposition 3.2]{Co}), and the unitarity of the $\gamma$-factor, as in the proof of Lemma \ref{archimedean-factors} one also gets that $\|\tilde{W}_0\|_{\tilde{\pi}_0}=\|W_0\|_{\pi_0}$.

We claim that the second summand in \eqref{s-derivative} is of bounded size. Note that again the invariance of the unitary inner product implies that the inner integral is equal to $\|\tilde{W}_0\|^2_{\tilde{\pi}_0}$. Thus using Cauchy's integral formula we can write the second summand as
$$\|\tilde{W}_0\|^2_{\tilde{\pi}_0}\int_{|s|=\epsilon}\frac{1}{s^2}\int_{\R^{r-1}}\overline{\hat{f}_{1/2}}\hat{f}_{1/2-s}\left[\begin{pmatrix}\mathrm{I}_{r-1}&\\c&1\end{pmatrix}\right]dc\frac{ds}{\pi i},$$
for some sufficiently small $\epsilon>0$. To show that the above integrals converge we start with the Iwasawa decomposition of $\begin{pmatrix}\mathrm{I}_{r-1}&\\c&1\end{pmatrix}$. One can check by induction or otherwise that there exists a $\tilde{z}(c)\in\R^{\times}$ so that
\begin{multline}\label{iwasawa-coordinate}
    \begin{pmatrix}\mathrm{I}_{r-1}&\\c&1\end{pmatrix}=\tilde{z}(c)\tilde{n}(c)\begin{pmatrix}\tilde{a}(c)&\\&1\end{pmatrix}\tilde{k}(c);\quad \tilde{n}(c)\in N(\R),\tilde{k}(c)\in K_\infty,\\\tilde{a}(c):=\diag(a_1(c),\dots,a_{r-1}(c));\quad a_i(c):=\frac{\sqrt{1+c_1^2+\dots+c_{i-1}^2}}{\sqrt{1+c_1^2+\dots+c_i^2}\sqrt{1+|c|^2}}.
\end{multline}
Thus using the transformation property \eqref{transformation-f} we get
$$\overline{\hat{f}_{1/2}}\hat{f}_{1/2-s}\left[\begin{pmatrix}\mathrm{I}_{r-1}&\\c&1\end{pmatrix}\right]\le (1+|c|^2)^{-r/2(1+\Re(s))}\|\hat{f}_{1/2}\hat{f}_{1/2-s}\|_{L^\infty(K_\infty)}.$$
Thus the second summand of \eqref{s-derivative} is bounded by
$$\ll_{\epsilon, f, W_0}\int_{\R^{r-1}}(1+|c|^2)^{-r/2(1-\epsilon)}dc.$$
The above integral is convergent for sufficiently small $\epsilon$.

We now focus on the third summand in \eqref{s-derivative}. In the inner integral we use Iwasawa coordinates for $h=ak$, move the $K$-integral outside, and make the change of variables $c\mapsto ck$ to rewrite it as
\begin{multline*}
    -\int_{\mathrm{O}_{r-1}(\R)}\int_{\R^{r-1}}|\hat{f}_{1/2}|^2\left[\begin{pmatrix}\mathrm{I}_{r-1}&\\ck&1\end{pmatrix}\right]\\\int_{A_{r-1}(\R)}|\tilde{W}_{0}|^2\left[\begin{pmatrix}a&\\&1\end{pmatrix}\begin{pmatrix}\mathrm{I}_{r-1}&\\cX&1\end{pmatrix}\begin{pmatrix}k&\\&1\end{pmatrix}\right]\log|\det(a)|\frac{da}{\delta(a)}dcdk
\end{multline*}
We use the Iwasawa decomposition of $\begin{pmatrix}\mathrm{I}_{r-1}&\\cX&1\end{pmatrix}$ as in \eqref{iwasawa-coordinate} to write it as $\tilde{n}(cX)\begin{pmatrix}\tilde{a}(cX)&\\&1\end{pmatrix}\tilde{k}(cX)$. Then using the left $N(\R)$-invariance of $|\tilde{W}_0|^2$ and changing varible $a\mapsto a\times \tilde{a}(cX)^{-1}$ we obtain that the above quantity is equal to
\begin{multline*}
    \int_{\mathrm{O}_{r-1}(\R)}\int_{\R^{r-1}}|\hat{f}_{1/2}|^2\left[\begin{pmatrix}\mathrm{I}_{r-1}&\\ck&1\end{pmatrix}\right]\log|\det(\tilde{a}(cX))|\delta(\tilde{a}(cX))\\\int_{A_{r-1}(\R)}|\tilde{W}_{0}|^2\left[\begin{pmatrix}a&\\&1\end{pmatrix}\tilde{k}(cX)\begin{pmatrix}k&\\&1\end{pmatrix}\right]\frac{da}{\delta(a)}dcdk\\
    -\int_{\mathrm{O}_{r-1}(\R)}\int_{\R^{r-1}}|\hat{f}_{1/2}|^2\left[\begin{pmatrix}\mathrm{I}_{r-1}&\\ck&1\end{pmatrix}\right]\delta(\tilde{a}(cX))\\\int_{A_{r-1}(\R)}|\tilde{W}_{0}|^2\left[\begin{pmatrix}a&\\&1\end{pmatrix}\tilde{k}(cX)\begin{pmatrix}k&\\&1\end{pmatrix}\right]\log|\det(a)|\frac{da}{\delta(a)}dcdk.
\end{multline*}
We write the above as $A-B$ where $A$ denotes the first term and $B$ denotes the second term above. To analyze $A$ we reverse engineer the above process: make the change of variables $a\mapsto a\times \tilde{a}(cX)$, use the left $N(\R)$-invariance of $|\tilde{W}_0|^2$, and make the change of variables $c\mapsto ck^{-1}$ to obtain
\begin{multline*}
    A=\int_{\R^{r-1}}|\hat{f}_{1/2}|^2\left[\begin{pmatrix}\mathrm{I}_{r-1}&\\c&1\end{pmatrix}\right]\int_{\mathrm{O}_{r-1}(\R)}\log|\det(a(ck^{-1}X))|\\\int_{A_{r-1}(\R)}|\tilde{W}_{0}|^2\left[\begin{pmatrix}ak&\\&1\end{pmatrix}\begin{pmatrix}\mathrm{I}_{r-1}&\\cX&1\end{pmatrix}\right]\frac{da}{\delta(a)}dkdc.
\end{multline*}
But $\det(\tilde{a}(cX))=(1+X^2|c|^2)^{-r/2}=\det(a(ck^{-1}X))$ for all $k\in\mathrm{O}_{r-1}(\R)$. Using that we can move the integral over $\mathrm{O_{r-1}(\R)}$ to couple with the integral over $A_{r-1}(\R)$ to obtain an integral over $N_{r-1}(\R)\backslash\GL_{r-1}(\R)$. Then once again appealing to the invariance of the unitary product we obtain
$$A=\|\tilde{W}_0\|^2_{\tilde{\pi}_0}\int_{\R^{r-1}}|\hat{f}_{1/2}|^2\left[\begin{pmatrix}\mathrm{I}_{r-1}&\\c&1\end{pmatrix}\right]\log|\det(\tilde{a}(cX))|dc.$$
Note that
\begin{align*}
    &\log|\det(\tilde{a}(cX))|=-\frac{r}{2}\log(1+X^2|c|^2)\\
    &=-\frac{r}{2}\log(1+X^2)+O(\log(1+|c|^2))=-r\log X+O_\epsilon((1+|c|)^\epsilon).
\end{align*}
Using the Iwasawa decomposition and transformation property of $\hat{f}_{1/2}$ as in \eqref{transformation-f}, similar to the second case we obtain
\begin{align*}
    &A+r\log X\|\tilde{W}_0\|^2_{\tilde{\pi}_0}\int_{\R^{r-1}}|\hat{f}_{1/2}|^2\left[\begin{pmatrix}\mathrm{I}_{r-1}&\\c&1\end{pmatrix}\right]dc\\
    &\ll_{\tilde{W}_0,\epsilon}\|\hat{f}_{1/2}\|^2_{L^\infty(K_\infty)}\int_{\R^{r-1}}(1+|c|^2)^{-r/2+\epsilon}dc\ll_{\Phi,\tilde{W}_0}1.
\end{align*}
Working as in the proof of Lemma \ref{archimedean-factors} we check that
$$\|\tilde{W}_0\|^2_{\tilde{\pi}_0}\int_{\R^{r-1}}|\hat{f}_{1/2}|^2\left[\begin{pmatrix}\mathrm{I}_{r-1}&\\c&1\end{pmatrix}\right]dc=\|W_0f_{1/2}\|^2.$$
Thus we obtain
\begin{equation}
    A=-r\log X\|W_0f_{1/2}\|^2+O_{W_0,f}(1).
\end{equation}

Now we prove that $B$ is of bounded size.
To prove that we first claim that 
$$\int_{A_{r-1}(\R)}|\tilde{W}_{0}|^2\left[\begin{pmatrix}a&\\&1\end{pmatrix}\tilde{k}(cX)\begin{pmatrix}k&\\&1\end{pmatrix}\right]\log|\det(a)|\frac{da}{\delta(a)}\ll_{\tilde{W}_0}1,$$
uniformly in $c$. We assume the claim. Now note that
$$\delta(\tilde{a}(cX))=\frac{(1+X^2|c|^2)^{r/2-1}}{\prod_{i=1}^{r-2}(1+c_1^2X^2+\dots+c_i^2X^2)}\ll \frac{X^{r-2}(1+|c|^2)^{r/2-1}}{\prod_{i=1}^{r-2}(1+c_i^2X^2)}.$$
We use the Iwasawa decomposition and work as before. Using transformation of $\hat{f}_{1/2}$ as in \eqref{transformation-f} we thus get
\begin{align*}
    B&\ll_{W_0,f} \int_{\mathrm{O}_{r-1}(\R)}\int_{\R^{r-1}}(1+|ck|^2)^{-r/2}\frac{X^{r-2}(1+|c|^2)^{r/2-1}}{\prod_{i=1}^{r-2}(1+c_i^2X^2)}dcdk\\
    &\ll \int_{\R^{r-1}}\prod_{i=1}^{r-1}(1+c_i^2)^{-1}dc,
\end{align*}
which we obtain by noting that $|ck|=|c|$ for $k\in\mathrm{O}_{r-1}(\R)$ and making the change of variables $c_i\mapsto c_i/X$ for $i\le r-2$. It is easy to see that the above integral is convergent, which yields that
$$B=O_{W_0,f}(1).$$

Now to prove the claim above let $\omega:=\tilde{k}(cX)\begin{pmatrix}k&\\&1\end{pmatrix}\in K_\infty$ implicitly depending on $cX$.
Note that from Lemma \ref{bound-for-convergence} we get that
$$\tilde{\pi}_0(\omega)\tilde{W}_0\left[\begin{pmatrix}a(y)&\\&1\end{pmatrix}\right]\ll_{\epsilon,M,W_0}\delta^{1/2-\epsilon}(a(y))|\det(a(y))|^{1/2-\vartheta_0-\epsilon}\prod_{i=1}^{r-1}\min(1,|y_i|^{-M}).$$
Thus we obtain
\begin{align*}
    &\int_{A_{r-1}(\R)}|\tilde{\pi}_0(\omega)\tilde{W}_{0}|^2\left[\begin{pmatrix}a&\\&1\end{pmatrix}\right]\log|\det(a)|\frac{da}{\delta(a)}\\
    &\ll_{W_0,\eta,M}\int_{(\R^\times)^ {r-1}}\prod_{i=1}^{r-1}\min(1,|y_i|^{-M})|\det(a(y))|^{1-2\vartheta_0-\epsilon}(|\det(a(y))|^\epsilon+|\det(a(y))|^{-\epsilon})\prod_{i}d^\times y_i.
\end{align*}
Employing the bound of $\vartheta_0$ from the statement of Theorem \ref{main-theorem} we check that the above integral is convergent for large enough $M$ and sufficiently small $\epsilon>0$, which yields the claim. 
\end{proof}

\begin{remark}\label{explicate-constant-term-local}
In the very last estimate of the proof of Lemma \ref{dual-term} we can only prove that the integral of the Whittaker function is of bounded size. It is not clear to us if or how one can improve the estimate to be a constant plus a power saving error term. This would potentially explicate the constant term of the asymptotic expansion in Theorem \ref{main-theorem} with a power saving error term; see Remark \ref{explicate-constant-term}.
\end{remark}

\begin{lemma}\label{easy-term}
We have
\begin{equation*}
\partial_{s=0}\Psi_\infty(\overline{f_{1/2,\infty}}f_{1/2+s,\infty}(.x_\infty),W_{0,\infty},\overline{W_{0,\infty}})=(r-1)\log X\|W_{0,\infty}f_{1/2,\infty}\|^2+O_{W_{0,\infty},\Phi_\infty}(1),
\end{equation*}
as $X$ tends to infinity.
\end{lemma}

\begin{proof}
The proof of this lemma is very similar to (and easier than) the proof of Lemma \ref{dual-term}. We first write $\Psi(\overline{f_{1/2}}f_{1/2+s}(.x),W_{0},\overline{W_{0}})$ as
$$X^{(r-1)s}\int_{N_{r-1}(\R)\backslash\GL_{r-1}(\R)}\int_{\R^{r-1}}|{W}_{0}|^2\left[\begin{pmatrix}h&\\c/X&1\end{pmatrix}\right]\overline{{f}_{1/2}}{f}_{1/2+s}\left[\begin{pmatrix}\mathrm{I}_{r-1}&\\c&1\end{pmatrix}\right]dc|\det(h)|^s{dh}.$$
Note that the $s=0$ derivative in the the statement of this lemma can be computed exactly same as we did in the calculation of \eqref{s-derivative} of Lemma \ref{dual-term} and can be seen equal to
\begin{multline}\label{s-derivative-easy}
    (r-1)\log X\int_{\R^{r-1}}|{f}_{1/2}|^2\left[\begin{pmatrix}\mathrm{I}_{r-1}&\\c&1\end{pmatrix}\right]\int_{N_{r-1}(\R)\backslash\GL_{r-1}(\R)}|{W}_{0}|^2\left[\begin{pmatrix}h&\\c/X&1\end{pmatrix}\right]dhdc\\
    +\int_{\R^{r-1}}\overline{{f}_{1/2}}\partial_{s=0}{f}_{1/2+s}\left[\begin{pmatrix}\mathrm{I}_{r-1}&\\c&1\end{pmatrix}\right]\int_{N_{r-1}(\R)\backslash\GL_{r-1}(\R)}|{W}_{0}|^2\left[\begin{pmatrix}h&\\c/X&1\end{pmatrix}\right]dhdc\\
    +\int_{\R^{r-1}}|{f}_{1/2}|^2\left[\begin{pmatrix}\mathrm{I}_{r-1}&\\c&1\end{pmatrix}\right]\int_{N_{r-1}(\R)\backslash\GL_{r-1}(\R)}|{W}_{0}|^2\left[\begin{pmatrix}h&\\c/X&1\end{pmatrix}\right]\log|\det(h)|dhdc.
\end{multline}
Exactly as in the proof of Lemma \ref{dual-term}, we can check (e.g.\ changing $\hat{f}_{1/2-s}$ to $f_{1/2+s}$ and $\tilde{W}_0$ to $W_0$) that the first and second summands in \eqref{s-derivative-easy} are
$$(r-1)\log X\|W_0f_{1/2}\|^2$$
and $O_{W_0,f}(1)$, respectively. We claim that the third summand in \eqref{s-derivative-easy} is also $O_{W_0,f}(1)$, which yields the lemma.

From the relation between $f$ and $\Phi$ from \eqref{relation-f-phi} we write
$$f_{1/2}\left[\begin{pmatrix}\mathrm{I}_{r-1}&\\c&1\end{pmatrix}\right]=\int_{\R^\times}\Phi(t(c,1))|t|^{r/2}d^\times t.$$
Recall the choice of $\Phi$ in \S\ref{local-choice}. Support of $\Phi$ being on $B_\tau(0,\dots,0,1)$ implies that in the above integral $t\asymp 1$ and hence $c\ll 1$. Below we show that
$$\int_{N_{r-1}(\R)\backslash\GL_{r-1}(\R)}|{W}_{0}|^2\left[\begin{pmatrix}h&\\c/X&1\end{pmatrix}\right]\log|\det(h)|dh\ll_{W_0} 1,$$
which clearly implies our claim above. 

We write $h=ak$ in Iwasawa coordinates and let $\omega:=\begin{pmatrix}k&\\c/X&1\end{pmatrix}$. Note that as $k\in\mathrm{O}_{r-1}(\R)$ and $c/X\ll 1$ there exists a fixed compact set $\Omega\in G(\R)$ such that $\omega\in \Omega$ for all relevant $c$ and $k$. Thus it is enough to show that
$$\int_{A_{r-1}(\R)}|\pi_0(\omega)W_0|^2\left[\begin{pmatrix}a&\\&1\end{pmatrix}\right]\log|\det(a)|\frac{da}{\delta(a)}\ll_{W_0,\Omega} 1.$$
This can be done similarly as we did at the end of the proof of Lemma \ref{dual-term}.
\end{proof}

\section{Analysis of the Regularized Term in the Period Side}\label{analysis-regularized-period}
Let $s\in\C$ with sufficiently small $\Re(s)$. Recall the regularized Eisenstein series
$\tilde{E}_s$ from \eqref{regularized-eisenstein-series}. The main proposition of this section is the following.
\begin{prop}\label{regularized-period}
$\tilde{E}_s$ is holomorphic in a sufficiently small neighbourhood of $s=0$ and is integrable on $\X$.
\end{prop}

Note from the definition \eqref{regularized-eisenstein-series} that $\tilde{E}_s$ is holomorphic in a punctured neighbourhood of $s=0$. Thus it is enough to prove that $\tilde{E}_s$ is holomorphic at $s=0$. Recall the description of the poles of the maximal Eisenstein series in \S\ref{eisenstein-series}. We know, in particular, $\overline{\Eis(f_{1/2})}\Eis(f_{1/2+s})$ is holomorphic at $s=0$ and we thus only need to show the following. 

\begin{lemma}\label{holomorphicity-central-point}
For fixed $g\in\X$
$$\Eis(\overline{f_{1/2}}f_{1/2+s})(g)+\Eis(\overline{\tilde{f}_{1/2}}\tilde{f}_{1/2+s})(g)$$ is holomorphic at $s=0$.
\end{lemma}

\begin{proof}
Our argument is to show that the residues $R$ and $\tilde{R}$ (which are independent of $g$) at the simple poles at $s=0$ of $\Eis(\overline{f_{1/2}}f_{1/2+s})$ and $\Eis(\overline{\tilde{f}_{1/2}}\tilde{f}_{1/2+s})$, respectively, cancel each other.

Let $\phi_0$ be the cusp form as we have chosen in \S\ref{local-choice}. From \eqref{degenerate-term1} we get
\begin{align*}
    R\|\phi_0\|^2_{2}&=\mathrm{Res}_{s=0}\langle |\phi_0|^2, \Eis(\overline{f_{1/2}}f_{1/2+s})\rangle\\ &=\frac{\zeta(r/2)^2}{\zeta(r)}L(1,\pi_0,\Ad)
    \Psi_\infty(|f_{1/2,\infty}|^2,W_{0,\infty},\overline{W_{0,\infty}}).
\end{align*}
Similarly, from \eqref{degenerate-term2} we get
\begin{align*}
    \tilde{R}\|\phi_0\|^2_{2}&=\mathrm{Res}_{s=0}\langle |\phi_0|^2, \Eis(\overline{\tilde{f}_{1/2}}\tilde{f}_{1/2+s})\rangle \\
    &= -\frac{\zeta(r/2)^2}{\zeta(r)}L(1,\pi_0,\Ad)\Psi_\infty(|\tilde{f}_{1/2,\infty}|^2,W_{0,\infty},\overline{W_{0,\infty}}).
\end{align*}
From Lemma \ref{archimedean-factors} with $x_\infty =1$ (and the equation preceding Lemma \ref{archimedean-factors}) we conclude that the $\Psi_\infty$ factors in the above expressions of $R$ and $\tilde{R}$ are equal.
\end{proof}

Now we prove some preparatory lemmata to prove the intgrability of $\tilde{E}_s$ on $\X$. We actually show that $\tilde{E}_s$ is integrable in the Siegel domain $\mathbb{S}$ as in \eqref{siegel-domain}, which contains $\X$.
Let $g\in\mathbb{S}$ with $g=(g_\infty,k_f)$ where $g_\infty=n_\infty\begin{pmatrix}a(y_\infty)&\\&1\end{pmatrix}k_\infty\in G(\R)$ in Iwasawa coordinates and $k_f\in K_f:=\prod_{p<\infty}K_p$. As $g\in\mathbb{S}$ we have $y_{j,\infty}\gg 1$. We recall the quantities in Proposition \ref{fourier-expansion-eisenstein-series} from \S\ref{fourier-eisenstein-series}.

\begin{lemma}\label{bound-compact-part}
Suppose that $i<r$. Let $s\in \C$ be away from a pole of $M_i^0f_s$ with $|\Re(s)|<2$. Then
$$\|M_i^0f_s\|_{L^\infty(K)}\ll 1.$$
Further, let $\R^\times\times K_f\ni (y_\infty,1)=:y$. Then for all $k\in K$ and $s$ with $|\Re(s)|<2$
$$\sum_{q\in\Q^\times}W_{f_s}^i(qy,k)\ll_{N} |y_\infty|^{-N},$$
where the sum in the left hand side converges absolutely.
\end{lemma}

\begin{proof}
In this proof we assume that $\Phi\in\mathcal{S}(\A^r)$ is an arbitrary Schwartz function. We get that for $k\in K$
$$M_i^0f_s(k)=\int_{\A^{r-i}}\int_{\A^\times}(k.\Phi)(0,t,x)|t|^{rs-r+i}d^\times t\ll\int_{|x|,|t|\ll 1}|t|^{rs-r+i}d^\times t dx\ll_{K,\Re(s)}1,$$
if $\Re(s)$ is sufficiently large.
On the other hand, using the Tate functional equation and working similarly we obtain
$$M_i^0f_s(k)=\int_{\A^\times}\widehat{(k.\Phi)}^i(te_i)|t|^{r-i+1-rs}\ll_{K,\Re(s)} 1,$$
if $\Re(s)$ is sufficiently negative. Using the Phragm\'en--Lindel\"of convexity principle and the compactness of $K$ we deduce the first claim.

Let $z\in\A$ and $k\in K$. Following a similar computations after \eqref{y-asymptotic-fourier-coeff} in \S\ref{fourier-eisenstein-series} we get that
$$W^i_{f_s}(z,k)=\int_{\A^{r-i}}\int_{\A^\times}(k.\Phi)(0,t,x)\overline{\psi_0(zx_1/t)}|t|^{rs-r+i}d^\times t dx.$$
This converges absolutely if $\Re(s)$ is sufficiently large. 

We first concentrate on the $x_1$ integral. In the archimedean component of this integral we integrate by parts with respect to the $x_{1,\infty}$ variable. This yields that archimedean integral is bounded by $\ll_N|z_\infty|^{-N}|t_\infty|^{N}$ for all large $N$. 

In the $p$-adic component we note that compact support of $\Phi_p$ forces $x_{1,p}$ to vary over a compact space. This implies that the $p$-adic integral vanishes unless $|z_p/t_p|\ll 1$. However, the support condition of $\Phi_p$ ensures that $|t_p|\ll 1$, which in turn restricts $z_p$ to be of bounded size.

Thus we can analytically continue the integral representation of $W^i_{f_s}$ to $\Re(s)$ sufficiently negative, but fixed. Altogether, estimating the integrals as before we obtain if $\Re(s)\ge-2$ then for sufficiently large $N$ we have
$$W_{f_s}^i(z,k)\ll_{K,N}|z_\infty|^{-N}\prod_{p<\infty}\mathrm{char}_{|z_p|\ll 1}.$$
Thus for $q\in\Q^\times$ and $y$ as in the statement of this lemma we have 
$$W_{f_s}^i(qy,k)\ll_{K,N}|y_\infty q_\infty|^{-N},$$
if the denominator of $q$ is bounded; otherwise, the above is zero. Thus the sum over $q\in\Q^\times$ is absolutely convergent for sufficiently large $N$. We conclude using the compactness of $K$.
\end{proof}

\begin{lemma}\label{y-exponents-1}
Let $g\in\mathbb{S}$ and $s\in\C$ with sufficiently small $\Re(s)$. Then
$$\overline{\Eis(f_{1/2})(g)}\Eis(f_{1/2+s})(g)-\overline{f_{1/2}(g)}f_{1/2+s}(g)-\overline{\tilde{f}_{1/2}(g)}\tilde{f}_{1/2+s}(g)\ll\delta^{1-\eta}\left[\begin{pmatrix}a(y_\infty)&\\&1\end{pmatrix}\right],$$
for some $\eta>0$.
\end{lemma}

\begin{proof}
Recall that
$$\delta\left[\begin{pmatrix}a(y_\infty)&\\&1\end{pmatrix}\right]=\prod_{j=1}^{r-1}|y_{j,\infty}|^{j(r-j)}.$$
Note that $y_{j,\infty}\gg 1$ as $g\in\mathbb{S}$. Thus it is enough the show that the exponents of $|y_{j,\infty}|$ arising in the left hand side in the expression in the lemma are less than $j(r-j)$.

We recall \eqref{y-asymptotic-fourier-coeff} and for $1<i<r$ write
\begin{multline*}
    H^i_s(g):=\sum_{q\in\Q}M_i^qf_{1/2+s}(g)\\=M_i^{0}f_{1/2+s}(g)+\sum_{q\in\Q^\times}\psi_{i(q)}(n_1)\prod_{j=1}^{i-1}|y_j|^{(1/2+s)j}\prod_{j=i}^{r-1}|y_j|^{(1/2-s)(r-j)}W^i_{f_{1/2+s}}(qy_i,k),
\end{multline*}
where $n_1$ is a unipotent element as in \S\ref{eisenstein-series}.
Using Lemma \ref{bound-compact-part} we obtain that
\begin{equation*}
    H^i_s(g)\ll \prod_{j=1}^{i-1}|y_{j,\infty}|^{(1/2+\Re(s))j}\prod_{j=i}^{r-1}|y_{j,\infty}|^{(1/2-\Re(s))(r-j)}.
\end{equation*}
On the other hand, we similarly obtain
\begin{equation*}
    H^1_s(g):=\sum_{q\in\Q^\times}M_1^qf_{1/2+s}(g)\ll_{K,N} \prod_{j=1}^{r-1}|y_{j,\infty}|^{(1/2-\Re(s))(r-j)-N\delta_{j=1}}.
\end{equation*}
We also record that
$$f_{1/2+s}(g)\ll \prod_{j=1}^{r-1}|y_{j,\infty}|^{(1/2+\Re(s))j},$$
and
$$\tilde{f}_{1/2+s}(g)\ll \prod_{j=1}^{r-1}|y_{j,\infty}|^{(1/2-\Re(s))(r-j)}.$$
We use Lemma \ref{fourier-expansion-basic} to rewrite
$$\Eis(f_{1/2+s})(g)=f_{1/2+s}(g)+\tilde{f}_{1/2+s}(g)+H^1_{s}(g)+\sum_{1<i<r}H_s^i(g).$$
After multiplying $\overline{\Eis(f_{1/2})}$ and $\Eis(f_{1/2+s})$ using the above expression and subtracting the terms $\overline{f_{1/2}}f_{1/2+s}$, and $\overline{\tilde{f}_{1/2}}\tilde{f}_{1/2+s}$ we are left with the following type of terms whose bounds are given below:
$$\overline{f_{1/2}}\tilde{f}_{1/2+s}(g)\ll \prod_{j=1}^{r-1}|y_{j,\infty}|^{r/2-(r-j)\Re(s)}.$$
If we replace the left hand side above with $\overline{\tilde{f}_{1/2}}{f}_{1/2+s}(g)$ then a similar inequality holds with the exponent in the right hand side being $r/2+j\Re(s)$.
In any case for sufficiently small $\Re(s)$ and $r\ge 3$ we have
$$r/2+r|\Re(s)|<j(r-j),\quad 1\le j<r.$$
A similar estimate can be done for $\overline{f_{1/2}}H^1_s(g)$.
Next we check that
$$\overline{H^1_0(g)}\tilde{f}_{1/2+s}\ll \prod_{j=1}^{r-1}|y_{j,\infty}|^{(1-\Re(s))(r-j)-N\delta_{j=1}}.$$
A similar estimate can be obtained if we replace $\tilde{f}_{1/2+s}$ by $H^1_s$ on the left hand side above. Clearly, for sufficiently small $\Re(s)$ we have
$$(1-\Re(s))(r-j)-N\delta_{j=1}<j(r-j).$$
Finally, for $1<i<r$ the exponent of $y_{j,\infty}$ of $\overline{H^i_0}$ is $\le(r-2)/2$. On the other hand the same of $G_s$ is $\le(1/2+|\Re(s)|)(r-1)$ for $G_s$ being one of ${f_{1/2+s}}$, $\tilde{f}_{1/2+s}$, or $H^i_s$ with $i<r$. So the exponent of $y_{j,\infty}$ of the product $\overline{H^i_0}G_s$ for $1<i<r$ is 
$$\le (r-2+r-1)/2+(r-1)|\Re(s)|<j(r-j),$$
for sufficiently small $\Re(s)$. 

Similarly, one estimates remaining terms of the form $\overline{H^1_0}f_{1/2+s}$, $\overline{G_0}H^i_s$, and $\overline{\tilde{f}_{1/2}}H^1_s$, which we leave for the reader.
Hence we conclude the proof.
\end{proof}

\begin{lemma}\label{y-exponents-2}
Let $g\in\mathbb{S}$ and $s\in \C$ with sufficiently small $\Re(s)$. Then
$$\Eis(\overline{f_{1/2}}f_{1/2+s})(g)-\overline{f_{1/2}(g)}f_{1/2+s}(g)\ll \delta^{1-\eta}\left[\begin{pmatrix}a(y_\infty)&\\&1\end{pmatrix}\right],$$
and also
$$\Eis(\overline{\tilde{f}_{1/2}}\tilde{f}_{1/2+s})(g)-\overline{\tilde{f}_{1/2}(g)}\tilde{f}_{1/2+s}(g)\ll\delta^{1-\eta}\left[\begin{pmatrix}a(y_\infty)&\\&1\end{pmatrix}\right],$$
for some $\eta>0$.
\end{lemma}

\begin{proof}
We take a very similar path as in the proof of Lemma \ref{y-exponents-1}. Let $s\in\C$ be away from the poles of the relevant Eisenstein series and $\Re(s)$ be sufficiently small.

First note that $\overline{\tilde{f}_{1/2}}\tilde{f}_{1/2+s}\in \I_{1,r-1}(1-s)$. We use the functional equation of the Eisenstein series \cite[Proposition 2.1]{Co}: there exists $\tilde{F}_s\in \I_{r-1,1}(s)$ such that
$$\Eis(\overline{\tilde{f}_{1/2}}\tilde{f}_{1/2+s})=\Eis(\tilde{F}_s).$$
In fact, $\tilde{F}_s$ is the pre-image of $\overline{\tilde{f}_{1/2}}\tilde{f}_{1/2+s}$ under the standard intertwiner from $\I_{1,r-1}(1-s)$ to $\I_{r-1,1}(s)$, i.e.\
$$\overline{\tilde{f}_{1/2}}\tilde{f}_{1/2+s}=M_1^0\tilde{F}_s,$$
and in particular, $\tilde{F}_s$ is holomorphic in a sufficiently small neighbourhood of $s=0$.

From Lemma \ref{fourier-expansion-basic} we get that
\begin{align*}
&\Eis(\overline{\tilde{f}_{1/2}}\tilde{f}_{1/2+s})-\overline{\tilde{f}_{1/2}(g)}\tilde{f}_{1/2+s}(g)=\Eis(\tilde{F}_s)(g)-M_1^0\tilde{F}_s(g)\\
&=\tilde{F}_s(g)+\sum_{q\in\Q^\times}M_1^q\tilde{F}_s(g) +\sum_{1<i<r}\sum_{q\in\Q}M_i^q\tilde{F}_s(g).
\end{align*}
We now bound each summand above similarly to the proof of Lemma \ref{y-exponents-1}.
As $\tilde{F}_s\in\I_{r-1,1}(s)$ we obtain that
$$\tilde{F}_s(g)\ll \left|\tilde{F}_s\left[\begin{pmatrix}a(y_\infty)&\\&1\end{pmatrix}\right]\right|\ll  \prod_{j=1}^{r-1}|y_{j,\infty}|^{j\Re(s)}.$$
Here we applied $\|\tilde{F}_s\|_{L^\infty(K)}\ll 1$, which can be deduced similarly to the proof of Lemma \ref{bound-compact-part} and applying holomorphicity of $\tilde{F}_s$ for small $s$. 

On the other hand, once again recalling \eqref{y-asymptotic-fourier-coeff} we obtain for $1<i<r$ that
\begin{multline*}
    \sum_{q\in\Q}M_i^q\tilde{F}_s(g)
    \ll M_i^0\tilde{F}_s\left[\begin{pmatrix}a(y_\infty)&\\&1\end{pmatrix}\right]
    \\+\prod_{j=1}^{i-1}|y_{j,\infty}|^{j\Re(s)}\prod_{j=i}^{r-1}|y_{j,\infty}|^{(1-\Re(s))(r-j)}\sum_{q\in\Q^\times}\|W^i_{\tilde{F}_s}(qy_i,.)\|_{L^\infty(K)}\\
    \ll \prod_{j=1}^{i-1}|y_{j,\infty}|^{j\Re(s)}\prod_{j=i}^{r-1}|y_{j,\infty}|^{(1-\Re(s))(r-j)}.
\end{multline*}
In the last estimate above we used that
$$\sum_{q\in\Q^\times}\|W^i_{\tilde{F}_s}(qy_i,.)\|_{L^\infty(K)}\ll_N |y_{i,\infty}|^{-N},$$
which can be deduced similarly to the proof of Lemma \ref{bound-compact-part}.
Similarly, we deduce that
$$\sum_{q\in\Q^\times}M_1^q\tilde{F}_s(g)\ll_N \prod_{j=1}^{r-1}|y_{j,\infty}|^{(1-\Re(s))(r-j)-N\delta_{j=1}}.$$
In each case the exponent of $y_{j,\infty}$ is strictly smaller than $j(r-j)$, which concludes the proof for the second assertion for sufficiently small $\Re(s)$. The first assertion can be proved similarly (and more easily), which we leave for the reader.
\end{proof}

\begin{proof}[Proof of Proposition \ref{regularized-period}]
In Lemma \ref{holomorphicity-central-point} we already have proved holomorphicity of $\tilde{E}_s$ at $s=0$. From Lemma \ref{y-exponents-1} and Lemma \ref{y-exponents-2} we conclude by the triangle inequality that for sufficiently small $\Re(s)$ and $g\in\mathbb{S}$
$$\tilde{E}_s(g)\ll \delta^{1-\eta}\left[\begin{pmatrix}a(y_\infty)&\\&1\end{pmatrix}\right],$$
for some $\eta>0$. Thus
$$\int_\X |\tilde{E}_s|(g)dg\ll \int_{y_{j,\infty}\gg 1}\delta^{-\eta}\left[\begin{pmatrix}a(y_\infty)&\\&1\end{pmatrix}\right]\prod_jd^\times y_{j,\infty}.$$
The last integral is convergent and we conclude.
\end{proof}

\section{Analysis of the Spectral Side}\label{analysis-spectral-side}
Recall the spectral weight $J_X(\pi_\infty)$ from \eqref{normalized-spectral-weight}, the choices of the local components from \S\ref{local-choice}, and the $\vartheta_0$-temperedness assumption on $\pi_{0,\infty}$ from the statement of Theorem \ref{main-theorem}. In this section we prove the remaining second property of the spectral weight as described in Theorem \ref{main-theorem}. That is, we show that $J_X(\pi_\infty)$ is uniformly bounded away from zero if $\pi_\infty$ is $\vartheta$-tempered with $\vartheta+\vartheta_0<1/2$ and $C(\pi_\infty)<X$.

\begin{prop}\label{main-prop-spectral-side}
Let $\pi\in\hat{\X}_\gen$ be such that $\pi_\infty$ is $\vartheta$-tempered with $\vartheta+\vartheta_0<1/2$. Let $\pi_0$ be the cuspidal automorphic representation as in Theorem \ref{main-theorem}. Then,
$$J_X(\pi_\infty)\gg 1,\quad\text{if }C(\pi_\infty)<X,$$
where the implied constant possibly depends on $W_{0,\infty},\Phi_\infty$.
\end{prop}

For the rest of this section, to ease notations, we drop the $\infty$-subscript everywhere.

We recall the notations and definition of the Sobolev norm $\mathcal{S}_d$ as in \cite[\S2.3.2]{MV}, \cite[\S3.9]{JN}. Let $\{H\}$ be a basis of $\mathrm{Lie}(G(\R))$. We define a Laplacian on $G(\R)$ by
\begin{equation}\label{laplacian}
    \mathcal{D}:=1-\sum_{H}H^2,
\end{equation}
which is positive definite and self-adjoint on any unitary representation $\xi$ of $G(\R)$.
For any $v\in\xi$ we define the $d$-th Sobolev norm of $v$ by
$$\mathcal{S}_d(v):=\|\mathcal{D}^dv\|_\xi.$$
We refer to \cite[\S2.4]{MV} for a collection of useful properties of the Sobolev norm.

Let $W\in\pi$ be a unit vector such that in the Kirillov model $W$ is given by 
\begin{equation}\label{newvector-choice}
    W\left[\begin{pmatrix}g&\\&1\end{pmatrix}\right]:= W_0\left[\begin{pmatrix}g&\\&1\end{pmatrix}\right].
\end{equation}
Note that such a choice is valid due to the choice of $W_0$ in \S\ref{local-choice}. In fact, $W\in \pi$ is an analytic newvector in the sense of \S\ref{analytic-newvector}.

\begin{lemma}\label{bound-for-convergence}
Let $W_0$ be as in \S\ref{local-choice}. Let $A_{r-1}(\R)\mathrm{O}_{r-1}(\R)\ni h=ak$ as before. If $c\ll 1$ then
$$W_0\left[\begin{pmatrix}h&\\c/X&1\end{pmatrix}\right]\ll_{\eta,\pi_0} |\det(a)|^{-\vartheta_0}\delta^{1/2-\eta}\left[\begin{pmatrix}a&\\&1\end{pmatrix}\right]\min(1,a_{r-1}^{-M})\prod_{i=1}^{r-2}\min(1,(a_i/a_{i+1})^{-M}),$$
for any $\eta>0$.
\end{lemma}

This lemma is proved in \cite[Lemma 5.2]{JN} for $\pi_0$ being a tempered representation. Here we modify the proof to accommodate the $\vartheta_0$-tempered case.

\begin{proof}
Let $W_0':=\pi_0\left[\begin{pmatrix}k&\\c/X&1\end{pmatrix}\right]W_0$. Note that $k\in \mathrm{O}_{r-1}(\R)$ and $|c|\ll 1$ vary over compact sets. Hence,
it is enough to show that
$$W'_0\left[\begin{pmatrix}a&\\&1\end{pmatrix}\right]\ll_{\eta,M,\pi_0} |\det(a)|^{-\vartheta_0}\delta^{1/2-\eta}\left[\begin{pmatrix}a&\\&1\end{pmatrix}\right]\min(1,a_{r-1}^{-M})\prod_{i=1}^{r-2}\min(1,(a_i/a_{i+1})^{-M}).$$
We take a very similar path as in the proof of \cite[Lemma 5.2]{JN}.

We define $W_1:=d\pi_0(Y^M)(W_0')$, where $Y$ is a Lie algebra element such that
$$d\pi_0(Y)W_0'\left[\begin{pmatrix}a&\\&1\end{pmatrix}\right]=a_{r-1}W_0'\left[\begin{pmatrix}a&\\&1\end{pmatrix}\right].$$
Thus it is enough to prove
\begin{equation}\label{estimate-w1}
    W_1\left[\begin{pmatrix}a&\\&1\end{pmatrix}\right]\ll_{\eta,M,\pi_0} |\det(a)|^{-\vartheta_0}\delta^{1/2-\eta}\left[\begin{pmatrix}a&\\&1\end{pmatrix}\right]\prod_{i=1}^{r-2}\min(1,(a_i/a_{i+1})^{-M}).
\end{equation}
We use the Dixmier--Malliavin Lemma (see \cite{DM}) to find finitely many $\alpha_i\in C_c^\infty(G(\R))$ and $W_i\in\pi_0^\infty$ such that
$$W_1=\sum_i \pi_0(\alpha_i)W_i.$$
Thus to prove \eqref{estimate-w1} it is enough to show \eqref{estimate-w1} with $W_1$ replaced by $\pi_0(\alpha_i)W_i=:W_2$ for each $i$.

Let $\sigma\in \R$. We use the Whittaker--Plancherel formula to expand
\begin{equation}\label{expansion-w2}
    |\det(a)|^{-\sigma}W_2\left[\begin{pmatrix}a&\\&1\end{pmatrix}\right]
    =\int_{\widehat{\GL_{r-1}(\R)}}\sum_{W'\in \B(\pi')}W'(a)Z_{W_2,\sigma}(W')d\mu_\loc(\pi'),
\end{equation}
which is valid for $\sigma$ in some left half plane. Here 
\begin{align*}
&Z_{W_2,\sigma}(W'):=\int_{N_{r-1}(\R)\backslash\GL_{r-1}(\R)}W_2\left[\begin{pmatrix}h&\\&1\end{pmatrix}\right]\overline{W'(h)}|\det(h)|^{-\sigma}dh\\
&=\gamma(1/2-\sigma,\pi_0\otimes\overline{\pi'})^{-1}\omega_{\pi'}(-1)^{r-1}\int_{N_{r-1}(\R)\backslash\GL_{r-1}(\R)}\tilde{W}_2\left[\begin{pmatrix}h&\\&1\end{pmatrix}\right]\overline{\tilde{W}'(h)}|\det(h)|^{\sigma}dh.
\end{align*}
In the last line we have used the $\GL(r)\times \GL(r-1)$ local functional equation. Here $\gamma(.)$ denotes the local gamma factor and $\omega_{\pi'}$ denotes the central character of $\pi'$. Finally, $\tilde{W}$ denotes the contragredient of $W$ defined by $\tilde{W}(g):=W(wg^{-t})$ where $w$ is the long Weyl element of the respective group.

Let $\tilde{\alpha}_i(g):=\alpha_i(g^{-t})$. Let $N^*$ be the unipotent radical of upper triangular matrices attached to the partition $r=(r-1)+1$. Recalling that $W_2=\pi_0(\alpha_i)W_i$ we can write
\begin{align*}
    \tilde{W}_2\left[\begin{pmatrix}h&\\&1\end{pmatrix}\right]
    &=\int_{G(\R)}\tilde{\alpha}_i(g)\tilde{W}_i\left[\begin{pmatrix}h&\\&1\end{pmatrix}g\right]dg\\
    &=\int_{N^*\backslash G(\R)}\tilde{W}_i\left[\begin{pmatrix}h&\\&1\end{pmatrix}g\right]\int_{N^*}\tilde{\alpha}_i(n^*g)\overline{\psi_{e_{r-1}h}(n^*)}dn^* dg,
\end{align*}
where $e_{r-1}$ is the row vector $(0,\dots,0,1)$.
Then we have that
\begin{multline*}
Z_{W_2,\sigma}(W')=\gamma(1/2-\sigma,\pi_0\otimes\overline{\pi'})^{-1}\omega_{\pi'}(-1)^{r-1}\\
\int_{N^*\backslash G(\R)}\int_{N_{r-1}(\R)\backslash\GL_{r-1}(\R)}\tilde{W}_i\left[\begin{pmatrix}h&\\&1\end{pmatrix}g\right]\int_{N^*}\tilde{\alpha}_i(n^*g)\overline{\psi_{e_{r-1}h}(n^*)}dn^* \overline{\tilde{W}'(h)}|\det(h)|^{\sigma}dh dg.
\end{multline*}
We choose an orthonormal basis $\B(\pi')$ consisting of eigenfunctions of the Laplacian $\D'$ on $\GL_{r-1}(\R)$, as defined in \eqref{laplacian}, and integrate by parts the $h$-integral $L$ times with respect to $\D'$. We note that $W'\otimes|\det|^\sigma$ is also an eigenfunction of $\D'$.
We recall a bound of the gamma factor from \cite[Lemma 3.1]{JN}:
$$\gamma(1/2-\sigma,\pi_0\otimes\overline{\pi'})\ll_{\sigma,\pi_0} C(\pi')^{r\sigma}.$$
We apply the Cauchy--Schwarz on the above $h$ integral. Then we use the above bound of the gamma factor and unitarity of $\pi_0$ to obtain that
\begin{multline*}
    Z_{W_2,\sigma}(W')\ll C(\pi')^{r\sigma}\lambda_{\tilde{W}'}^{-L}\int_{N^*\backslash G(\R)}\left(\int_{N_{r-1}(\R)\backslash\GL_{r-1}(\R)}\left|{\D'}^{L}\left(\int_{N^*}\alpha_i(n^*g)\overline{\psi_{e_{r-1}h}(n^*)}dn^*\right)\right.\right.\\ \left.\left.\overline{\tilde{W}'(h)}|\det(h)|^{\sigma}\right|^2dh\right)^{1/2}dg,
\end{multline*}
where $\lambda_{\tilde{W}'}$ is the $\D'$-eigenvalue of $\tilde{W}'$.
The above $N^*$-integral gives rise to a Schwartz function in $e_{r-1}h$, which can be seen integrating by parts several times in the $N^*$-integral. Thus
$${\D'}^{L}\left(\int_{N^*}\alpha_i(n^*g)\overline{\psi_{e_{r-1}h}(n^*)}dn^*\right)\ll\min(1,|e_{r-1}h|^{-N}).$$
Noting that $g$ varies over a compact set in $G(\R)$ modulo $N^*$ we obtain that
$$Z_{W_2,\sigma}(W')\ll C(\pi')^{r\sigma}\lambda_{\tilde{W}'}^{-L}\left(\int_{N_{r-1}(\R)\backslash\GL_{r-1}(\R)}\min(1,|e_{r-1}h|^{-N}) |{\tilde{W}'(h)}|^2|\det(h)|^{2\sigma}dh\right)^{1/2}.$$
We use \cite[Lemma 5.2]{JN} on $\tilde{W}'$ (which is in the tempered representation $\tilde{\pi}'$) to check that the above integral is absolutely convergent for any $\sigma>0$. In particular, from the location of the first pole of $\gamma(1/2-\sigma,\pi_0\otimes\overline{\pi'})^{-1}$ we may conclude that one can choose $\sigma$ in $(0,1/2-\vartheta_0)$ in the definition of $Z_{W_2,\sigma}(W')$.

Again we use \cite[Lemma 5.2]{JN} to estimate $W'(a)$ in \eqref{expansion-w2} by
$$\ll \delta^{1/2-\eta}(a)\prod_{i=1}^{r-2}\min(1,(a_i/a_{i+1})^{-M})\lambda_{W'}^d,$$
where $d$ only depends on $M$. We choose $\sigma=1/2-\vartheta_0-\eta$ to obtain that
\begin{multline*}
W_2\left[\begin{pmatrix}a&\\&1\end{pmatrix}\right]\ll_{\eta,M,\pi_0} |\det(a)|^{-\vartheta_0}\delta^{1/2-\eta}\left[\begin{pmatrix}a&\\&1\end{pmatrix}\right]\\
\prod_{i=1}^{r-2}\min(1,(a_i/a_{i+1})^{-M})\int_{\widehat{\GL_{r-1}(\R)}}C(\pi')^{r\sigma}\sum_{W'\in \B(\pi')}\lambda_{W'}^{d-L}d\mu_\loc(\pi').
\end{multline*}
We make $L$ sufficiently large and invoke \cite[Lemma 3.3]{JN} to conclude that the above sum and integral are absolutely convergent.
\end{proof}

\begin{lemma}\label{newvector-theorem}
Let $W$ be as in \eqref{newvector-choice} and $W_0$ be as in \S\ref{local-choice}. Let $V$ be $W$ or $W_0$ and $\xi$ be $\pi$ or $\pi_{0}$, respectively. Also let $A_{r-1}(\R)\mathrm{O}_{r-1}(\R)\ni h=ak$ where $a=\diag(a_1,\dots,a_{r-1})$, and $|c|\ll 1$. Then for any sufficiently small $\eta>0$
$$V\left[\begin{pmatrix}h&\\c/X&1\end{pmatrix}\right]-V\left[\begin{pmatrix}h&\\&1\end{pmatrix}\right]\ll_{\eta} |\det(a)|^{-\theta}\delta^{1/2-\eta}\left[\begin{pmatrix}a&\\&1\end{pmatrix}\right]
\frac{C(\xi)|c|}{X}.$$
Here $\theta$ is $\vartheta$ or $\vartheta_0$ depending on whether $\xi$ is $\pi$ or $\pi_0$, respectively.
\end{lemma}

This is essentially the main result of analytic newvectors, proved in \cite[Proposition 4.1]{JN}, but in a more quantitative form. We need to only modify the proof of \cite[Proposition 4.1]{JN} and we describe that here.

\begin{proof}
Let $\sigma\in \R$ be in some left half plane. As in the proof of \cite[Proposition 4.1]{JN} we write the difference in the lemma 
as 
\begin{multline*}
    \int_{\widehat{\GL_{r-1}(\R)}}\omega_{\bar{\pi'}}((-1)^{r-1}C(\xi)^{-1})C(\xi)^{(r-1)\sigma}\gamma(1/2-\sigma,\xi\otimes\bar{\pi}')^{-1}\sum_{W'\in\B(\pi')}W'(h)|\det(h)|^{\sigma} \\
    \int_{N_{r-1}(\R)\backslash \GL_{r-1}(\R)}(e(cw't^{-1}e_1C(\xi)/X)-1)V\left[\begin{pmatrix}C(\xi)&\\&t\end{pmatrix}w\right]\overline{W'(tw')}|\det(t)|^{-\sigma} dtd\mu_\loc(\pi'),
\end{multline*}
where $w'$ is the long Weyl element of $\GL(r-1)$.
Note that $\pi'$ is tempered. We now use \cite[Lemma 5.2]{JN} for tempered representations to estimate
$$W'(h)\ll \delta^{1/2-\eta}(a)
S_{d}(W'),$$
for any $\eta>0$ and some $d>0$.
We choose $\sigma=1/2-\theta-\eta$ (which is admissible) and proceed as in the proof of \cite[Proposition 4.1]{JN} to conclude.
\end{proof}

\begin{proof}[Proof of Proposition \ref{main-prop-spectral-side}]
Recall the definition of $J_X$ from \eqref{normalized-spectral-weight} and \eqref{spectral-weight}. In the expression of \eqref{spectral-weight} we choose a basis $\B(\pi)$ containing an analytic newvector $W$ as in \eqref{newvector-choice}.
To show the required lower bound of $J_X$ it is enough to drop all but the term containing $W$ from the sum in \eqref{spectral-weight} by positivity and show that
$$X^{r-1}|\Psi(f_{1/2}(.x),W_0,\overline{W})|^2\gg 1,$$
if $C(\pi)<X$.

First, using \eqref{relation-f-phi} and the choices of the local components as in \S\ref{local-choice} we get
$$f_{1/2}\left[\begin{pmatrix}\mathrm{I}_{r-1}&\\c&1\end{pmatrix}\right]=\int_{\R^\times}\Phi(t(c,1))|t|^{r/2}d^\times t\ge 0.$$
The support condition of $\Phi$ in \S\ref{local-choice} implies that the above vanishes unless $|c|< \tau$. We use Bruhat coordinates and make the change of variables to write $X^{\frac{r-1}{2}}\Psi(f_{1/2}(.x),W_0,\overline{W})$ as
$$\int_{N_{r-1}(\R)\backslash\GL_{r-1}(\R)}\int_{\R^{r-1}}W_0\left[\begin{pmatrix}h&\\c/X&1\end{pmatrix}\right]\overline{W\left[\begin{pmatrix}h&\\c/X&1\end{pmatrix}\right]}f_{1/2}\left[\begin{pmatrix}\mathrm{I}_{r-1}&\\c&1\end{pmatrix}\right]dc \frac{dh}{|\det(h)|^{1/2}}.$$
We use Lemma \ref{newvector-theorem} for $W$, noting that $|c|<\tau$ and $C(\pi)<X$, to obtain the above integral is
\begin{multline}\label{using-coordinates}
    \int_{N_{r-1}(\R)\backslash\GL_{r-1}(\R)}\overline{W\left[\begin{pmatrix}h&\\&1\end{pmatrix}\right]}\int_{\R^{r-1}}W_0\left[\begin{pmatrix}h&\\c/X&1\end{pmatrix}\right]f_{1/2}\left[\begin{pmatrix}\mathrm{I}_{r-1}&\\c&1\end{pmatrix}\right]dc \frac{dh}{|\det(h)|^{1/2}}\\
    + O_\eta\left(\tau\int_{N_{r-1}(\R)\backslash \GL_{r-1}(\R)}\delta^{1/2-\eta}\left[\begin{pmatrix}h&\\&1\end{pmatrix}\right]\int_{\R^{r-1}}\left|W_0\left[\begin{pmatrix}h&\\c/X&1\end{pmatrix}\right]f_{1/2}\left[\begin{pmatrix}\mathrm{I}_{r-1}&\\c&1\end{pmatrix}\right]\right|dc \frac{dh}{|\det(h)|^{1/2+\vartheta}}\right).
\end{multline}
We use Lemma \ref{newvector-theorem} for $W_0$ and the definition of $W$ in the Kirillov model as in \eqref{newvector-choice} to obtain that the main term of \eqref{using-coordinates} is equal to
\begin{multline}\label{main-term-using-coordinates}
    \int_{N_{r-1}(\R)\backslash\GL_{r-1}(\R)}|W_0|^2\left[\begin{pmatrix}h&\\&1\end{pmatrix}\right]\frac{dh}{|\det(h)|^{1/2}}\int_{\R^{r-1}}f_{1/2}\left[\begin{pmatrix}\mathrm{I}_{r-1}&\\c&1\end{pmatrix}\right]dc\\
    +O_{\pi_0}\left(\frac{1}{X}\int_{N_{r-1}(\R)\backslash\GL_{r-1}(\R)}|W_0|\left[\begin{pmatrix}h&\\&1\end{pmatrix}\right]\delta^{1/2-\eta}\left[\begin{pmatrix}h&\\&1\end{pmatrix}\right]\frac{dh}{|\det(h)|^{1/2+\vartheta_0}}\right.
    \\\left.\int_{\R^{r-1}}|f_{1/2}|\left[\begin{pmatrix}\mathrm{I}_{r-1}&\\c&1\end{pmatrix}\right]dc\right).
\end{multline}
From the choice of $\Phi$ in \S\ref{local-choice} we obtain that 
$$0\le\int_{\R^{r-1}}f_{1/2}\left[\begin{pmatrix}\mathrm{I}_{r-1}&\\c&1\end{pmatrix}\right]dc\asymp 1.$$
Also, the choice of $W_0$ in \S\ref{local-choice} ensures that
$$\int_{N_{r-1}(\R)\backslash\GL_{r-1}(\R)}|W_0|^2\left[\begin{pmatrix}h&\\&1\end{pmatrix}\right]\frac{dh}{|\det(h)|^{1/2}}\asymp_{\pi_0} 1.$$
So the main term of \eqref{main-term-using-coordinates} is
$$\asymp_{\pi_0}\int_{\R^{r-1}}f_{1/2}\left[\begin{pmatrix}\mathrm{I}_{r-1}&\\c&1\end{pmatrix}\right]dc.$$
On the other hand, the error term in \eqref{main-term-using-coordinates} is trivially $\ll_{\pi_0,\tau}X^{-1}$, which follows from the support condition of $W_0$ as in \S\ref{local-choice}. In total we obtain that \eqref{main-term-using-coordinates}, which is the main term of \eqref{using-coordinates}, is
$$\asymp_{\pi_0}\int_{\R^{r-1}}f_{1/2}\left[\begin{pmatrix}\mathrm{I}_{r-1}&\\c&1\end{pmatrix}\right]dc +O_{\pi_0,\tau}(1/X).$$
Now we focus on the error term of \eqref{using-coordinates}. We use Iwasawa coordinates in the integral and use Lemma \ref{bound-for-convergence} to estimate the error term by
\begin{multline*}
    \ll_{M} \tau\int_{\R^{r-1}}f_{1/2}\left[\begin{pmatrix}\mathrm{I}_{r-1}&\\c&1\end{pmatrix}\right]dc\int_{A_{r-1}(\R)}|\det(a)|^{1/2-\vartheta-\vartheta_0}\delta^{-2\eta}\left[\begin{pmatrix}h&\\&1\end{pmatrix}\right]\\
    \min(1,a_{r-1}^{-M})\prod_{i=1}^{r-2}\min(1,(a_i/a_{i+1})^{-M}){d^\times a}.
\end{multline*}
We recall the assumption that $\vartheta+\vartheta_0<1/2$. Hence, the inner integral is convergent for sufficiently small $\eta$ and large enough $M$. Thus we obtain \eqref{using-coordinates} is $$\asymp_{\pi_0} (1+\tau O_{\pi_0}(1))\int_{\R^{r-1}}f_{1/2}\left[\begin{pmatrix}\mathrm{I}_{r-1}&\\c&1\end{pmatrix}\right]dc+O_{\pi_0,\tau}(1/X).$$
We conclude that the above is $\gg 1$ by making $\tau$ sufficiently small but fixed.
\end{proof}

\end{document}